\documentclass[10pt,twoside]{amsart}
\usepackage{mathtools,amsmath,amssymb}
\usepackage{mathrsfs}

\usepackage[a4paper,top=1in, bottom=1in, left=1in, right=1in]{geometry} 
\usepackage{graphicx} 
\usepackage[colorlinks=true,linkcolor=blue,citecolor=black]{hyperref} 
\usepackage{cleveref}
\usepackage[backend=biber,backref=true,url=false,maxnames=9,isbn=false,sorting=nyt,style=numeric-comp]{biblatex} 
\renewbibmacro*{pageref}{%
  \iflistundef{pageref}
    {}
    {\addspace
     \mkbibbrackets{\printlist[pageref][-\value{listtotal}]{pageref}}}} 
\usepackage{lipsum}
\usepackage{adjustbox} 
\usepackage{tikz-cd}
\usepackage{mathptmx}
\usepackage{quiver}
\usepackage{float}
\usepackage{amssymb}
\usepackage[english]{babel}
\usepackage{amsthm}
\usepackage{amsmath}
\usepackage{mathtools}
\usepackage{soul}
\usepackage{nicematrix}
\usepackage{array}   
\usepackage{booktabs}
\usepackage{float}
\usepackage{orcidlink}
\usepackage[toc,page]{appendix}
\usepackage{caption}
\usepackage{mathrsfs}     
\makeatletter
\renewcommand*\env@matrix[1][*\c@MaxMatrixCols c]{%
  \hskip -\arraycolsep
  \let\@ifnextchar\new@ifnextchar
  \array{#1}}
\makeatother

\usepackage{makecell, caption, chngcntr} %

\usepackage{siunitx} %
\counterwithin{table}{section}

\newcommand{\GL}{\operatorname{GL}}

\newcommand{\PSL}{\operatorname{PSL}}

\newcommand{\SL}{\operatorname{SL}}
\newcommand{\Q}{\mathbb Q}
\newcommand{\Gal}{\mathrm{Gal}}

\newcommand{\Z}{\mathbb Z}
\newcommand{\irr}{\mathrm{Irr}}
\newcommand{\Irr}{\mathrm{Irr}}

\newcommand{\M}{\mathrm{M}}
\newcommand{\cd}{\mathrm{cd}}
\newcommand{\lin}{\mathrm{lin}}
\newcommand{\nl}{\mathrm{nl}}

\usepackage{thmtools, thm-restate}
\declaretheorem{theorem}
\newtheorem{lemma}{Lemma}[section]

\newtheorem{corollary}[lemma]{Corollary}
\theoremstyle{definition} 

\newtheorem{example}[lemma]{Example}
\newtheorem{remark}[lemma]{Remark}

\title[On rational representations and rational group algebra of $\operatorname{GL}_2(q)$]{On rational representations and rational group algebra of $\operatorname{GL}_2(q)$}
\date{\today}
\author[R.~K.~Choudhary]{Ram Karan Choudhary\orcidlink{0009-0003-7688-6397}}
\email[(Choudhary)]{ramkchoudhary1997@gmail.com, ram.choudhary@iiserpune.ac.in}
\address{Indian Institute of Science Education and Research Pune, Dr. Homi Bhabha Road, Pashan, Pune 411008, India}
\author[S.~K.~Prajapati]{Sunil Kumar Prajapati}
\email[(Prajapati)]{skprajapati@iitbbs.ac.in}
\address{Indian Institute of Technology, Bhubaneswar, Arugul Campus, Jatni, Khurda-752050, India}
\thanks{The first author is supported by an institute postdoctoral fellowship from IISER Pune.}
\date{\today}
\dedicatory{}
\begin{document}
\begin{abstract}
In this article, we study rational representations of $G=\operatorname{GL}_2(q)$, where $q$ is a prime power. Let $\rho$ be an irreducible representation of $G$ over $\mathbb{Q}$. Then $\rho$ affords the character
\[
\Omega(\chi)=m_{\mathbb{Q}}(\chi)\sum_{\sigma\in\operatorname{Gal}(\mathbb{Q}(\chi)/\mathbb{Q})}\chi^{\sigma},
\]
for some irreducible complex character $\chi$ of $G$, where $m_{\mathbb{Q}}(\chi)$ denotes the Schur index of $\chi$ over $\mathbb{Q}$, and conversely, every character of this form is afforded by an irreducible representation of $G$ over $\mathbb{Q}$. We obtain a combinatorial description for the counting of inequivalent irreducible $\mathbb{Q}$-representations of $G$ of each distinct degree. Furthermore, we briefly determine the rational character table of $G$ and present a method for constructing an irreducible rational matrix representation $\rho$ of $G$ affording the character $\Omega(\chi)$, where $\chi$ is an irreducible complex character of $G$ arising from parabolic induction. Finally, using the results on the rational representations of $G$, we derive an explicit combinatorial formula, depending only on $q$, for the Wedderburn decomposition of $\mathbb{Q}G$.
\end{abstract}

\maketitle
\tableofcontents
\section{Introduction}\label{sec:intro}
The construction of matrix representations of finite groups over a field is a classical and fundamental problem in mathematics. The study of representations of finite groups over $\mathbb{C}$ was initiated by Frobenius, while Schur and Brauer later extended the theory to subfields of $\mathbb{C}$, particularly $\mathbb{R}$ and $\mathbb{Q}$. Although a substantial body of literature is devoted to the computation of matrix representations, most existing methods are confined to certain types of representations for specific classes of groups. For instance, an algorithm for computing an irreducible complex matrix representation affording a character $\chi$ of a finite group $G$, under the condition that $\chi(1)\leq 100$, is presented in \cite{Dabbaghian, Dixon} and implemented in the \textsc{Repsn} package of {\sf GAP} \cite{Gap}. Nevertheless, determining all inequivalent irreducible matrix representations of a finite group over a field $\mathbb{F}$, including the case $\mathbb{F}=\mathbb{C}$, remains a difficult and fundamental problem in representation theory.

Let $\mathbb{F}_q$ be the finite field of order $q$, where $q=p^m$ for some prime $p$ and positive integer $m$. Denote by $\GL_2(q)$ the group of all invertible $2\times 2$ matrices over $\mathbb{F}_q$, namely,
\[
\GL_2(q)=\{A\in \mathrm{M}_2(\mathbb{F}_q)\;:\;\det(A)\neq 0\}.
\]
In recent years, rational-valued irreducible complex characters of finite groups have attracted considerable attention (see \cite{Grittini, Navarro, Navarro1, Navarro2, Tiep}). The present paper is devoted to the construction of irreducible matrix representations of the finite classical group $G=\GL_2(q)$ over $\mathbb{Q}$. The study of matrix representations of a finite group $G$ over $\mathbb{Q}$ is important for several reasons. One of the central problems in rationality theory is to determine whether an $\mathbb{F}$-representation of $G$ can be realized over a subfield, such as realizing a $\mathbb{C}$-representation over $\mathbb{R}$ or $\mathbb{Q}$, and to understand the action of the Galois group on vector spaces over $\mathbb{Q}$. Moreover, rational representations admit integral realizations by a theorem of Burnside \cite{Burnside}, which states that if $\rho \colon G \to \GL_n(\mathbb{Q})$ is a representation, then there exists a conjugate representation $\tilde{\rho}$ such that $\tilde{\rho}(g)\in \GL_n(\mathbb{Z})$ for all $g\in G$. Thus, all matrix entries of $\tilde{\rho}$ are integers. In addition, a $\mathbb{Z}G$-module $M$ is $\mathbb{Z}$-reducible if and only if $\mathbb{Q}M$ is reducible as a $\mathbb{Q}G$-module (see \cite[Theorem~73.9]{Curtis-Reiner}). Consequently, rational representations play an important role in various algebraic and arithmetic settings. Representations of finite groups over $\mathbb{Q}$ arise naturally in many areas of mathematics, much like complex representations, further emphasizing their significance (for a detailed account of rational representations of finite groups, see \cite{Bartel, Bouc, Ford, Kletzing, Plesken}).

Throughout this paper, $G$ denotes a finite group and $\Irr(G)$ denotes the set of all irreducible complex characters of $G$. For $\chi\in \Irr(G)$, define
\[
\Omega(\chi)= m_{\mathbb{Q}}(\chi)\sum_{\sigma \in \Gal(\mathbb{Q}(\chi)/\mathbb{Q})}\chi^{\sigma},
\]
where $m_{\mathbb{Q}}(\chi)$ denotes the Schur index of $\chi$ over $\mathbb{Q}$. Observe that $\Omega(\chi)$ is the character of an irreducible $\mathbb{Q}$-representation $\rho$ of $G$. Conversely, if $\rho$ is an irreducible $\mathbb{Q}$-representation of $G$, then there exists $\chi\in \Irr(G)$ such that $\Omega(\chi)$ is the character afforded by $\rho$ (see \cite[Corollary~10.2]{I}).

The character table of $\GL_2(q)$ is well known and was explicitly determined in~\cite{St51b}. We adopt the notation from \cite[p.~104]{GeMa20}. A detailed description of the character table of $\GL_2(q)$ is provided in Section~\ref{sec:conj-char}. The group $\GL_2(q)$ has $q^2-1$ irreducible characters, which are partitioned into the following four families:
\[
\chi^{(n)}_1, \quad
\chi^{(n)}_q, \quad
\chi^{(m,n)}_{q+1},
\quad \text{and} \quad
\chi^{(n)}_{q-1}.
\]

The number of isomorphism classes of irreducible representations of a finite group $G$ over $\mathbb{Q}$ equals the number of conjugacy classes of cyclic subgroups of $G$ (see \cite[Corollary~1, Chapter~13]{Serre}). We prove Theorem~\ref{thm:RationRepsGL2}, which determines the number of pairwise non-isomorphic simple $\mathbb{Q}G$-modules of each possible dimension for $G=\GL_2(q)$. Furthermore, we briefly describe the rational character table of $G$, consisting of the rational conjugacy classes of $G$ and the corresponding rational character values (see Section~\ref{sec:RatCharTab}).

\begin{theorem} \label{thm:RationRepsGL2}
Let $G = \GL_2(q)$, where $q$ is a prime power. For a positive integer $m$, let $\varphi(m)$ denote the number of integers in $[1,m]$ that are coprime to $m$, and let $\tau(m)$ denote the number of positive divisors of $m$. 
Then the following statements hold.

\begin{enumerate}
    \item For each $d \mid (q-1)$, there is exactly one (up to isomorphism) simple $\mathbb{Q} G$-module of dimension $\varphi(d)$ that affords the character $\Omega(\chi^{(n)}_1)$ for some $n$ such that $d = \dfrac{q-1}{\gcd(n,q-1)}$.

    \item For each $d \mid (q-1)$, there is exactly one (up to isomorphism) simple $\mathbb{Q} G$-module of dimension $q\,\varphi(d)$ that affords the character $\Omega(\chi^{(n)}_q)$ for some $n$ such that $d = \dfrac{q-1}{\gcd(n,q-1)}$.

    \item For each $d \mid (q^2-1)$ with $d \nmid (q-1)$, there is exactly one (up to isomorphism) simple $\mathbb{Q} G$-module of dimension $\frac{q-1}{2}\,\varphi(d)$ that affords the character $\Omega(\chi^{(n)}_{q-1})$ for some $n$ such that $d = \dfrac{q^2-1}{\gcd(n,q^2-1)}$.

    \item Let $U = (\mathbb{Z}/(q-1)\mathbb{Z})^\times$ act on the set $X$ of unordered pairs $(m,n)$ with $0 \le m, n \le q-2$ and $m \neq n$ by multiplication by multiplication modulo $q-1$. Let $\mathcal{O}$ be a set of representatives of the orbits of this action. For each $(m,n) \in \mathcal{O}$, there is exactly one (up to isomorphism) simple $\mathbb{Q} G$-module of dimension $(q+1)|\mathcal O_{(m,n)}|$
    that affords the character $\Omega(\chi^{(m,n)}_{q+1})$, where 
    $\mathcal O_{(m,n)}$ denotes the orbit corresponding to $(m,n) \in \mathcal{O}$. 

    \item $|\irr_{\mathbb{Q}}(G)| = \tau(q-1) + \tau(q^2-1) + |\mathcal{O}|$.
\end{enumerate}
\end{theorem}

For a finite $p$-group $G$ and $\chi \in \Irr(G)$, the construction of an irreducible rational matrix representation affording the character $\Omega(\chi)$ is equivalent to determining a pair $(H,\psi)$, where $H$ is a subgroup of $G$ and $\psi \in \lin(H)$ satisfies $\psi^G=\chi$ and $\mathbb{Q}(\psi)=\mathbb{Q}(\chi)$ when $p$ is an odd prime. In the case $p=2$, one requires a subgroup $H\leq G$ and a character $\psi \in \Irr(H)$ such that $\psi^G=\chi$, $\mathbb{Q}(\chi)=\mathbb{Q}(\psi)$, and $H/\ker(\psi)$ is isomorphic to a cyclic, generalized quaternion, dihedral, or semi-dihedral group. Such a pair $(H,\psi)$ is called a \emph{required pair} for an irreducible rational matrix representation of $G$ affording the character $\Omega(\chi)$ (see~\cite{Ford} for the existence of such pairs). The reader may refer to \cite[Section~3]{Ram1} for an explicit procedure to compute an irreducible rational matrix representation of $G$ affording the character $\Omega(\chi)$ using a required pair.

Suppose $\rho : G \to \mathrm{GL}_d(\mathbb{C})$ is an irreducible complex representation of $G$ affording the character $\chi$. Let $\mathbb{L} = \mathbb{Q}(\chi)$, and choose a $\mathbb{Q}$-basis $\{b_1,\dots,b_t\}$ of $\mathbb{L}$, where $t=[\mathbb{L}:\mathbb{Q}]$. In the special case when $m_{\mathbb{Q}}(\chi)=1$, there exists a conjugate representation $\rho' : G \to \mathrm{GL}_d(\mathbb{L})$ of $\rho$. For each $g\in G$, write $\rho'(g) = (a_{ij}(g))$. Thus, $a_{ij}(g)\in \mathbb{L}$. Replace each entry $a_{ij}(g)$ by the $t\times t$ matrix $\M(a_{ij}(g))$ over $\mathbb{Q}$ representing multiplication by $a_{ij}(g)$ on $\mathbb{L}$ with respect to the basis $\{b_j\}$. Then the block matrix
\[
\widetilde{\rho}(g)=\bigl(\M(a_{ij}(g))\bigr)_{i,j=1}^{d}
\]
has size $dt$ and entries in $\mathbb{Q}$. Consequently, $\widetilde{\rho}$ defines an irreducible $\mathbb{Q}$-representation of $G$ affording the character $\Omega(\chi)$. However, there is no general systematic procedure for constructing such a $\rho'$.

The study of matrix representations of several classes of finite $p$-groups using required pairs has been carried out by us as part of the first author’s PhD work (see \cite{Ram1, Ram4, Ram6}). The present work extends this line of investigation. In this article, we establish analogous results by employing techniques similar to those used in the theory of required pairs for constructing irreducible matrix representations of finite $p$-groups over $\Q$. More precisely, we construct irreducible matrix representations of $G=\GL_2(q)$ over $\Q$ affording the character $\Omega(\chi)$, where $\chi \in \Irr(G)$ is not afforded by a cuspidal representation of $G$ (see Section~\ref{sec:MatrixReps}). Furthermore, we also prove the existence of irreducible rational representations $\rho$ and $\theta$ of $\GL_2(q)$, both of degree greater than one, such that their tensor product $\rho \otimes_{\mathbb{Q}} \theta$ remains irreducible over $\mathbb{Q}$ (see Corollary~\ref{cor:tensor}). In contrast, no such phenomenon occurs for irreducible complex representations of $\GL_2(q)$.

Next, this article also investigates the Wedderburn decomposition of the rational group algebra of $\GL_2(q)$, which is obtained using the results from its rational representations. Let $G$ be a group and let $\mathbb{F}$ be a field. Then the group algebra $\mathbb{F}G$ is the free $\mathbb{F}$-module with basis $G$. The description of group algebras is a classical problem in algebra and admits several approaches. In the semisimple case, namely when $G$ is finite and the characteristic of $\mathbb{F}$ does not divide $|G|$, the problem reduces to describing the Wedderburn components of the group algebra. By the Wedderburn--Artin theorem, a semisimple group algebra $\mathbb{F}G$ decomposes as a direct sum of matrix algebras over division rings
\[
\mathbb{F}G \cong \bigoplus_{i=1}^{r} \M_{n_i}(D_i),
\]
where each $\M_{n_i}(D_i)$ is a full matrix algebra of size $n_i$ over a division ring $D_i$ that is finite-dimensional over its center. These algebras are called the \emph{simple components} of $\mathbb{F}G$. Furthermore, by the Brauer--Witt theorem \cite{Yam}, each simple component is Brauer equivalent to a cyclotomic algebra. The Wedderburn decomposition of rational group algebras has been extensively studied because of its significance in understanding various algebraic structures (see \cite{Herman, Jes-Rio, Rit-Seh}).

Such decompositions can be determined explicitly either through character-theoretic methods or by using subgroup structures via Shoda pairs for certain classes of rational group algebras, particularly those associated with monomial groups. We refer to some of the recent works on rational group algebras based on Shoda pair theory presented in~\cite{Shoda1, Shoda2, Shoda3, Shoda4, Shoda5, Shoda6}. The reader may also see the survey article by Bakshi~\cite{Bakshi} for a comprehensive account of the literature on Shoda pair theory. Computational methods, such as those implemented in the \textsc{Wedderga} package of {\sf GAP}~\cite{Gap}, make use of these techniques to compute the Wedderburn decomposition of rational group algebras of such groups. Nevertheless, explicit computations remain difficult, especially for groups of large order.

For finite abelian groups, Perlis and Walker \cite{PW} derived a combinatorial formula for the Wedderburn decomposition of their rational group algebras by counting cyclic subgroups. Rational group algebras of alternating groups have also been studied, and an explicit formula for the decomposition of $\Q A_n$ is given in \cite[Theorem~2]{GiambrunoJesper1998}. For several classes of finite $p$-groups, we developed combinatorial techniques to determine explicitly the Wedderburn decomposition of their rational group algebras, either by reducing the problem to abelian sections or by using parameters arising from group presentations (see \cite{Ram1, Ram2, Ram3, Ram4, Ram5, Ram6}). Furthermore, for $G\in\{\SL_2(q), \PSL_2(q)\}$, combinatorial formulas for the Wedderburn decomposition of $\Q G$ have been established in \cite{Ram7}.

For a positive integer $d$, let $\zeta_d$ denote a primitive $d$-th root of unity in $\mathbb{C}$. Theorem~\ref{thm:wedderburn GL} establishes a combinatorial formula, depending only on $q$, for the Wedderburn decomposition of the rational group algebra of $\GL_2(q)$, where $q$ is a prime power.

\begin{theorem}\label{thm:wedderburn GL}
Let $G = \GL_2(q)$, where $q$ is a prime power. Let $U = (\mathbb{Z}/(q-1)\mathbb{Z})^\times$ act on the set $X$ of unordered pairs $(m,n)$ with $0 \le m, n \le q-2$ and $m \neq n$ by multiplication modulo $q-1$. Let $\mathcal{O}$ be a set of representatives of the orbits of this action. Then the Wedderburn decomposition of $\Q G$ is given by
\[
\Q G \cong \bigoplus_{d \mid q-1} \Q(\zeta_d)
\bigoplus_{\substack{d \mid q^2-1 \\ d \nmid q-1}}
\M_{q-1}(\Q(\zeta_d + \zeta_d^{\,q}))
\bigoplus_{d \mid q-1} \M_q(\Q(\zeta_d))
\bigoplus_{(m,n)\in\mathcal{O}}
\M_{q+1}(\Q(\zeta_{q-1}^{\,m+n},\;\zeta_{q-1}^{\,m}+\zeta_{q-1}^{\,n})).
\]
\end{theorem}

We now introduce the main notation used throughout the paper. For a positive integer $m$, let $\varphi(m)$ denote the number of integers in $[1,m]$ that are coprime to $m$, let $\tau(m)$ denote the number of positive divisors of $m$, and let $\zeta_m$ denote a primitive $m$-th root of unity in $\mathbb{C}$. For a finite group $G$, we denote by $\irr(G)$, $\lin(G)$, and $\nl(G)$ the sets of irreducible complex characters of $G$, linear complex characters of $G$, and non-linear irreducible complex characters of $G$, respectively. We write $\cd(G)=\{\chi(1)\mid \chi\in \irr(G)\}$ for the set of character degrees of $G$. For $\chi\in \irr(G)$, let $m_{\mathbb{Q}}(\chi)$ denote the Schur index of $\chi$ over $\mathbb{Q}$, and let $\mathbb{Q}(\chi)$ denote the field extension of $\mathbb{Q}$ generated by the values $\chi(g)$ for all $g\in G$. We further define
$\Omega(\chi)=m_{\mathbb{Q}}(\chi)\sum_{\sigma\in \Gal(\mathbb{Q}(\chi)/\mathbb{Q})}\chi^\sigma$.
We denote by $\irr_{\mathbb{Q}}(G)$ the set of irreducible rational characters of $G$. Occasionally, we write $\chi_\rho$ for the character afforded by a representation $\rho$ of $G$. All other notation used in this article is standard.

The organization of the article is as follows. In Section~\ref{sec:conj-char}, we introduce the character table of $G=\GL_2(q)$, which will be used throughout the article. In Section~\ref{sec:RationalReps}, we study the simple $\mathbb{Q}G$-modules and prove Theorem~\ref{thm:RationRepsGL2}. In Section~\ref{sec:RatCharTab}, we briefly discuss the rational character table of $G$. In Section~\ref{sec:MatrixReps}, we present a method for constructing irreducible matrix representations of $G$ over $\mathbb{Q}$. Finally, in Section~\ref{sec:GroupAlgbera}, we prove Theorem~\ref{thm:wedderburn GL}.

\section{Character table of $\GL_2(q)$}\label{sec:conj-char}
In this section, we describe the character table of $\GL_2(q)$. The character table of $\GL_2(q)$ is classical and was explicitly determined in~\cite{St51b}. We follow the notation used in \cite[p.~104]{GeMa20}. Although the material in this section is standard, we include it for the convenience of the reader, since information from the character table of this family of groups is used extensively throughout the article.

Let $G=\GL_2(q)$, where $q$ is an arbitrary prime power. The conjugacy classes of $\GL_2(q)$ are described via normal forms of matrices. Let $\sigma$ be a generator of $\mathbb{F}_q^\times$, and let $\tau$ be a generator of $\mathbb{F}_{q^2}^\times$. There are four types of conjugacy classes, with representatives given as follows:
\[
A_1(a):=
\begin{pmatrix}
\sigma^a & 0\\
0 & \sigma^a
\end{pmatrix},
\quad
\text{where } 0\le a\le q-2,
\]
\[
A_2(a):=
\begin{pmatrix}
\sigma^a & 0\\
1 & \sigma^a
\end{pmatrix},
\quad
\text{where } 0\le a\le q-2,
\]
\[
A_3(a,b):=
\begin{pmatrix}
\sigma^a & 0\\
0 & \sigma^b
\end{pmatrix},
\qquad
\text{where } 0\le a<b\le q-2,
\]
\[
\text{and }
B_1(a):=
\begin{pmatrix}
0 & -\tau^{a(q+1)}\\
1 & \tau^a+\tau^{aq}
\end{pmatrix},
\quad
\text{where } a\in E_q \text{ and } (q+1)\nmid a.
\]

Here, $E_q\subseteq \{0,1,\dots,q^2-2\}$ denotes a set of representatives for the equivalence relation on $\mathbb{Z}$ defined by
\[
a\sim a'
\quad \text{if} \quad
a\equiv a' \pmod{q^2-1}
\quad \text{or} \quad
a\equiv qa' \pmod{q^2-1}.
\]

Observe that the matrix $B_1(a)$ is diagonalizable over $\mathbb{F}_{q^2}$, though not over $\mathbb{F}_q$, with eigenvalues $\tau^a$ and $\tau^{aq}$. The group $G=\GL_2(q)$ has exactly $q^2-1$ conjugacy classes. Consequently, it also has precisely $q^2-1$ irreducible complex characters, which fall into the following four families:
\[
\chi^{(n)}_1
\quad \text{of degree } 1,
\qquad
0\le n\le q-2,
\]
\[
\chi^{(n)}_q
\quad \text{of degree } q,
\qquad
0\le n\le q-2,
\]
\[
\chi^{(m,n)}_{q+1}
\quad \text{of degree } q+1,
\qquad
0\le n<m\le q-2,
\]
\[
\text{and }
\chi^{(n)}_{q-1}
\quad \text{of degree } q-1,
\qquad
n\in E_q \text{ and } (q+1)\nmid n.
\]

Finally, let $\varepsilon,\eta\in \mathbb{C}$ be primitive roots of unity of orders $q-1$ and $q^2-1$, respectively. Table~\ref{table-gl(2)} gives the character table of $G=\GL_2(q)$, as presented in~\cite[p.~104]{GeMa20}.

\begin{table}[h]
\centering
\caption{The character table of $G=\GL_2(q)$}\label{table-gl(2)}
\[
\begin{array}{c|cccc}
& A_1(a) & A_2(a) & A_3(a,b) & B_1(a)\\
\hline
\chi^{(n)}_1
& \varepsilon^{2na}
& \varepsilon^{2na}
& \varepsilon^{n(a+b)}
& \varepsilon^{na}
\\[1ex]
\chi^{(n)}_q
& q\varepsilon^{2na}
& 0
& \varepsilon^{n(a+b)}
& -\varepsilon^{na}
\\[1ex]
\chi^{(m,n)}_{q+1}
& (q+1)\varepsilon^{(m+n)a}
& \varepsilon^{(m+n)a}
& \varepsilon^{ma+nb}+\varepsilon^{na+mb}
& 0
\\[1ex]
\chi^{(n)}_{q-1}
& (q-1)\eta^{na(q+1)}
& -\eta^{na(q+1)}
& 0
& -(\eta^{na}+\eta^{naq})
\end{array}
\]
\end{table}

\section{Proof of Theorem~\ref{thm:RationRepsGL2}} \label{sec:RationalReps}
In this section, we prove Theorem~\ref{thm:RationRepsGL2}, which determines the number of pairwise non-isomorphic simple $\Q G$-modules of each possible dimension for $G=\GL_2(q)$, where $q$ is a power of a prime $p$. To establish the theorem, we first develop the necessary background material. 

Before proceeding to the proof of Theorem~\ref{thm:RationRepsGL2}, we begin with Lemma~\ref{lemma:CharacterField}, which describes the character field $\Q(\chi)$ for each $\chi\in \irr(G)$. Here, $\Q(\chi)$ denotes the field obtained by adjoining to $\Q$ all character values $\chi(g)$, where $g\in G$. Throughout this section, we retain the notation for the characters introduced in Section~\ref{sec:conj-char}.

\begin{lemma}\label{lemma:CharacterField}
Let $G=\GL_2(q)$, where $q$ is a prime power, and let $\chi\in \irr(G)$ such that
\[
\chi \in \left\{\,
\chi^{(n)}_1,\,
\chi^{(n)}_q,\,
\chi^{(m,n)}_{q+1}\ (0\le n<m\le q-2),\,
\chi^{(n)}_{q-1}\ (n\in E_q,\ (q+1)\nmid n)
\,\right\}.
\]
Let $\varepsilon$ be a primitive $(q-1)$-th root of unity in $\mathbb{C}$, and let $\eta$ be a primitive $(q^2-1)$-th root of unity in $\mathbb{C}$. Then the character fields $\Q(\chi)$ are given as follows.
\begin{enumerate}
\item $\Q(\chi^{(n)}_1)=\Q(\varepsilon^{n})$.
\item $\Q(\chi^{(n)}_q)=\Q(\varepsilon^{n})$.
\item $\Q(\chi^{(m,n)}_{q+1})=\Q(\varepsilon^{m+n},\, \varepsilon^{m}+\varepsilon^{n})$.
\item $\Q(\chi^{(n)}_{q-1})=\Q(\eta^{n}+\eta^{nq})$.
\end{enumerate}
\end{lemma}

\begin{proof}
Let $G=\GL_2(q)$, where $q$ is a prime power, and let
\[
\chi \in \left\{\,
\chi^{(n)}_1,\,
\chi^{(n)}_q,\,
\chi^{(m,n)}_{q+1}\ (0\le n<m\le q-2),\,
\chi^{(n)}_{q-1}\ (n\in E_q,\ (q+1)\nmid n)
\,\right\}.
\]

\begin{enumerate}
    \item Let $\chi=\chi^{(n)}_1$. Observe that $$\chi^{(n)}_1(B_1(1))=\varepsilon^{n}.$$ Moreover, every value of $\chi^{(n)}_1$ is a power of $\varepsilon^{n}$ (see Table~\ref{table-gl(2)}). Hence, $\Q(\chi^{(n)}_1)=\Q(\varepsilon^{n})$.

    \item Let $\chi=\chi^{(n)}_q$. We have
    \[
    \chi^{(n)}_q(B_1(1))=-\varepsilon^{n}.
    \]
    Furthermore, all values of $\chi^{(n)}_q$ belong to $\Q(\varepsilon^{n})$ (see Table~\ref{table-gl(2)}). Therefore, $\Q(\chi^{(n)}_q)=\Q(\varepsilon^{n})$.

    \item From Table~\ref{table-gl(2)}, we obtain
\[
\chi^{(m,n)}_{q+1}(A_1(1))=(q+1)\varepsilon^{m+n}
\quad \text{and} \quad
\chi^{(m,n)}_{q+1}(A_3(0,1))=\varepsilon^{m}+\varepsilon^{n}.
\]
Hence,
\[
\Q(\varepsilon^{m+n},\,\varepsilon^{m}+\varepsilon^{n})
\subseteq
\Q(\chi^{(m,n)}_{q+1}).
\]

For the reverse inclusion, consider the value of $\chi^{(m,n)}_{q+1}$ on $A_3(a,b)$
\[
\chi^{(m,n)}_{q+1}(A_3(a,b))
=
\varepsilon^{ma+nb}+\varepsilon^{na+mb},
\]
where $0\le a<b\le q-2$ (see Table~\ref{table-gl(2)}). Set $x=\varepsilon^{m}$ and $y=\varepsilon^{n}$. Then we have
\begin{align*}
\chi^{(m,n)}_{q+1}(A_3(a,b))
&=
\varepsilon^{ma+nb}+\varepsilon^{na+mb}\\
&=
x^{a}y^{b}+x^{b}y^{a}\\
&=
(xy)^{\min(a,b)}(x^{|a-b|}+y^{|a-b|}).
\end{align*}
Since $xy=\varepsilon^{m+n}$, it remains to show that $x^{k}+y^{k}$ belongs to
\[
\mathbb{K}:=\Q(\varepsilon^{m+n},\,\varepsilon^{m}+\varepsilon^{n})
\]
for every integer $k\ge0$.

Observe that $x$ and $y$ are the roots of the polynomial
\[
T^2-(\varepsilon^{m}+\varepsilon^{n})T+\varepsilon^{m+n}=0.
\]
Thus, $x$ and $y$ are algebraic over $\mathbb{K}$. For each integer $k\ge0$, let
\[
S_k=x^k+y^k.
\]
Since $S_k$ is a symmetric polynomial in $x$ and $y$ with integer coefficients, the Fundamental Theorem of Symmetric Polynomials implies that $S_k$ can be expressed as a polynomial in the elementary symmetric functions
\[
x+y=\varepsilon^{m}+\varepsilon^{n}
\quad \text{and} \quad
xy=\varepsilon^{m+n}.
\]
Equivalently, the sequence $(S_k)$ satisfies the recurrence relation
\[
S_k=(x+y)S_{k-1}-xy\,S_{k-2}
\qquad (k\ge2),
\]
with initial conditions $S_0=2$ and $S_1=x+y$. Hence, every $S_k$ belongs to $\mathbb{K}$. Consequently,
\[
\varepsilon^{ma+nb}+\varepsilon^{na+mb}\in \mathbb{K},
\]
and therefore all values of $\chi^{(m,n)}_{q+1}$ lie in $\mathbb{K}$. This shows that
\[
\Q(\chi^{(m,n)}_{q+1})
\subseteq
\mathbb{K}.
\]
Hence, we get
\[
\Q(\chi^{(m,n)}_{q+1})
=
\Q(\varepsilon^{m+n},\,\varepsilon^{m}+\varepsilon^{n}).
\]

\item Let $g\in G$, and let $\eta$ be a primitive $(q^2-1)$-th root of unity in $\mathbb{C}$. From Table~\ref{table-gl(2)}, we have
\[
\chi^{(n)}_{q-1}(g)
\in
\left\{
0,\,
(q-1)\eta^{na(q+1)},\,
-\eta^{na(q+1)},\,
-(\eta^{na}+\eta^{naq})
\right\},
\]
where $n\in E_q$, $(q+1)\nmid n$, and $E_q\subseteq \{0,1,\dots,q^2-2\}$ is a set of representatives for the equivalence relation on $\mathbb{Z}$ defined by
\[
a\sim a'
\quad \text{if} \quad
a\equiv a' \pmod{q^2-1}
\quad \text{or} \quad
a\equiv qa' \pmod{q^2-1}.
\]

Let
\[
d=\frac{q^2-1}{\gcd(n,q^2-1)}.
\]
Then $\eta^n$ is a primitive $d$-th root of unity. Define
\[
\sigma:\eta^n\longmapsto \eta^{nq}.
\]
Since $\gcd(q,q^2-1)=1$ and $q^2\equiv1\pmod d$, the map $\sigma$ defines a Galois automorphism of $\Q(\eta^n)$. Moreover, because $(q+1)\nmid n$, we have $\sigma\neq \operatorname{id}$, while $\sigma^2=\operatorname{id}$.
Therefore, the fixed field of $\sigma$ has index $2$ in $\Q(\eta^n)$.

Now, observe that
\[
\sigma(\eta^n+\eta^{nq})
=
\eta^{nq}+\eta^{nq^2}
=
\eta^{nq}+\eta^n
=
\eta^n+\eta^{nq},
\]
since $q^2\equiv1\pmod d$. Hence, $\eta^n+\eta^{nq}$ lies in the fixed field of $\sigma$. As the fixed field has degree $2$, we conclude that $\Q(\eta^n+\eta^{nq})$ is precisely the fixed field of $\sigma$.

Next, we note that $\sigma(\eta^{n(q+1)})
=
\sigma(\eta^n\eta^{nq})
=
\eta^{nq}\eta^{nq^2}
=
\eta^{nq}\eta^n
=
\eta^{n(q+1)}$. Therefore, we have 
\[
\chi^{(n)}_{q-1}(A_1(a))
=
(q-1)\eta^{na(q+1)}
\in
\Q(\eta^n+\eta^{nq}),
\]
and similarly,
\[
\chi^{(n)}_{q-1}(A_2(a))
=
-\eta^{na(q+1)}
\in
\Q(\eta^n+\eta^{nq}).
\]
Moreover, we have
\begin{align*}
\sigma(\chi^{(n)}_{q-1}(B_1(a)))
&=
\sigma\bigl(-(\eta^{na}+\eta^{naq})\bigr)\\
&=
-(\eta^{naq}+\eta^{naq^2})\\
&=
-(\eta^{naq}+\eta^{na})\\
&=
\chi^{(n)}_{q-1}(B_1(a)).
\end{align*}
Thus,
\[
\chi^{(n)}_{q-1}(B_1(a))
\in
\Q(\eta^n+\eta^{nq}).
\]
Since $\chi^{(n)}_{q-1}(B_1(1))
=
-(\eta^n+\eta^{nq})$,
the field generated by the character values is exactly $\Q(\eta^n+\eta^{nq})$. Therefore, we conclude that
\[
\Q(\chi^{(n)}_{q-1})
=
\Q(\eta^n+\eta^{nq}).
\]
\end{enumerate}
This completes the proof of Lemma~\ref{lemma:CharacterField}.
\end{proof}

 We now describe the Galois conjugacy classes. Two irreducible characters $\chi,\psi\in \irr(G)$ are said to be \emph{Galois conjugates} over $\Q$ if $\Q(\chi)=\Q(\psi)$ and there exists $\sigma\in \Gal(\Q(\chi)/\Q)$ such that
\[
\chi^\sigma=\psi,
\]
where $\chi^\sigma(g)=\sigma(\chi(g))$ for all $g\in G$. It is straightforward to verify that this defines an equivalence relation on $\Irr(G)$. In Lemma~\ref{lemma:galois}, we determine the number of distinct Galois conjugacy classes of irreducible complex characters of $\GL_2(q)$ over $\Q$. We begin with Lemma~\ref{SC}

\begin{lemma}\cite[Lemma~9.17]{I}\label{SC}
		Let $G$ be a finite group and $\chi \in \irr(G)$. Denote by $E(\chi)$ the Galois conjugacy class of $\chi$ over $\mathbb{Q}$. Then
		\[
		|E(\chi)| = [\mathbb{Q}(\chi) : \mathbb{Q}].
		\]
	\end{lemma}

\begin{lemma}\label{lemma:galois}
Let $G=\GL_2(q)$, where $q$ is a prime power, and let
\[
\chi \in \left\{\,
\chi^{(n)}_1,\,
\chi^{(n)}_q,\,
\chi^{(m,n)}_{q+1}\ (0\le n<m\le q-2),\,
\chi^{(n)}_{q-1}\ (n\in E_q,\ (q+1)\nmid n)
\,\right\}.
\]
For any integer $d\ge1$, let $\zeta_d$ denote a primitive $d$-th root of unity in $\mathbb{C}$. Then the distinct Galois conjugacy classes of irreducible complex characters of $G$ over $\Q$ are described as follows.

\begin{enumerate}
\item \textbf{Case $(\chi=\chi^{(n)}_1)$.}  
In this case, for each divisor $d$ of $q-1$, there is exactly one Galois conjugacy class having a representative $\chi$ such that $\Q(\chi)=\Q(\zeta_d)$. Consequently, this family of irreducible complex characters of $G$ decomposes into $\tau(q-1)$ distinct Galois conjugacy classes.

\item \textbf{Case $(\chi=\chi^{(n)}_q)$.}  
In this case, for each divisor $d$ of $q-1$, there is exactly one Galois conjugacy class having a representative $\chi$ such that $\Q(\chi)=\Q(\zeta_d)$. Hence, this family of irreducible complex characters of $G$ also splits into $\tau(q-1)$ distinct Galois conjugacy classes.

\item \textbf{Case $(\chi=\chi^{(m,n)}_{q+1})$.}  
Let $U=(\mathbb{Z}/(q-1)\mathbb{Z})^\times$ act on the set $X$ of unordered pairs $(m,n)$ with $0 \le m, n \le q-2$ and $m \neq n$ by multiplication modulo $q-1$. Let $\mathcal{O}$ be a set of representatives for the orbits of this action. For each $(m,n)\in \mathcal{O}$, there is exactly one Galois conjugacy class having a representative $\chi=\chi^{(m,n)}_{q+1}$ such that
\[
\Q(\chi)=\Q\bigl(\zeta_{q-1}^{\,m+n},\,\zeta_{q-1}^{\,m}+\zeta_{q-1}^{\,n}\bigr).
\]
Furthermore, if $q-1=\prod_{i=1}^r p_i^{e_i}$ is the prime factorisation of $q-1$, then the number of distinct Galois conjugacy classes of principal series characters of $G$ is
\[
\frac{1}{\phi(q-1)}
\sum_{k\in (\mathbb{Z}/(q-1)\mathbb{Z})^\times}
\left(
\binom{\gcd(k-1,q-1)}{2}
+
\frac{1}{2}
\bigl(
\gcd(k^2-1,q-1)-\gcd(k-1,q-1)
\bigr)
\right).
\]
In particular, this number depends only on the prime factorisation of $q-1$.

\item \textbf{Case $(\chi=\chi^{(n)}_{q-1})$.}  
In this case, for each divisor $d$ of $q^2-1$ satisfying $d\nmid(q-1)$, there is exactly one Galois conjugacy class having a representative $\chi$ such that $\Q(\chi)=\Q(\zeta_d+\zeta_d^{\,q})$. Therefore, the cuspidal characters of $G$ split into
\[
\tau(q^2-1)-\tau(q-1)
\]
distinct Galois conjugacy classes.
\end{enumerate}
\end{lemma}

\begin{proof}
We prove each part separately.
\begin{enumerate}
    \item Let $\chi=\chi^{(n)}_1$. By Lemma~\ref{lemma:CharacterField}, we have
    \[
    \Q(\chi)=\Q(\chi^{(n)}_1)=\Q(\varepsilon^n)=\Q(\zeta_d),
    \]
    where $d=\frac{q-1}{\gcd(n,q-1)}$, and $\varepsilon$ is a primitive $(q-1)$-th root of unity in $\mathbb{C}$. The Galois group
    \[
    \Gal(\Q(\zeta_d)/\Q)\cong(\mathbb{Z}/d\mathbb{Z})^\times
    \]
    acts transitively on the primitive $d$-th roots of unity. Hence, all characters in this family with character field $\Q(\zeta_d)$ belong to a single Galois orbit. Therefore, for each divisor $d$ of $q-1$, there is exactly one Galois conjugacy class having a representative $\chi$ such that $\Q(\chi)=\Q(\zeta_d)$. Consequently, this family splits into $\tau(q-1)$ distinct Galois conjugacy classes.

    \item The proof of this part is analogous to that of the previous case.

    \item By Lemma~\ref{lemma:CharacterField}, we have
    \[
    \Q(\chi^{(m,n)}_{q+1})
    =
    \Q(\varepsilon^{m+n},\varepsilon^m+\varepsilon^n)
    \subseteq
    \Q(\varepsilon),
    \]
    where $\varepsilon$ is a primitive $(q-1)$-th root of unity in $\mathbb{C}$. Observe that $\Q(\varepsilon)$ is the cyclotomic field of $(q-1)$-st roots of unity, and
    \[
    \Gal(\Q(\varepsilon)/\Q)
    \cong
    (\Z/(q-1)\Z)^\times,
    \]
    where $k\in(\Z/(q-1)\Z)^\times$ acts via
    \[
    \varepsilon\mapsto\varepsilon^k.
    \]

    Let $\chi=\chi^{(m,n)}_{q+1}$, and let $\sigma\in\Gal(\Q(\chi)/\Q)$. Since $\Q(\chi)\subseteq\Q(\varepsilon)$ and $\Q(\varepsilon)/\Q$ is Galois, the automorphism $\sigma$ extends to an automorphism $\widetilde{\sigma}$ of $\Q(\varepsilon)$. This extension corresponds to some $k\in(\Z/(q-1)\Z)^\times$ satisfying
    \[
    \widetilde{\sigma}(\varepsilon)=\varepsilon^k.
    \]
    From the character table, each value $\chi(g)$ is a polynomial in $\varepsilon^m$ and $\varepsilon^n$ with integer coefficients. Therefore,
    \[
    \sigma(\chi(g))
    =
    \widetilde{\sigma}(\chi(g))
    =
    \chi^{(km,kn)}_{q+1}(g),
    \]
    and hence
    \[
    \chi^\sigma=\chi^{(km,kn)}_{q+1}.
    \]

    Conversely, let $k\in(\Z/(q-1)\Z)^\times$. The automorphism of $\Q(\varepsilon)$ defined by $\varepsilon\mapsto\varepsilon^k$ sends
    \[
    \varepsilon^{m+n}\mapsto\varepsilon^{k(m+n)}
    \quad \text{and} \quad
    \varepsilon^m+\varepsilon^n
    \mapsto
    \varepsilon^{km}+\varepsilon^{kn}.
    \]
    Since $\Q(\chi^{(m,n)}_{q+1})=\Q(\varepsilon^{m+n},\varepsilon^m+\varepsilon^n)$,
    this automorphism restricts to an element of $\Gal(\Q(\chi)/\Q)$ sending $\chi^{(m,n)}_{q+1}$ to $\chi^{(km,kn)}_{q+1}$.

    Therefore, the Galois orbit of $\chi^{(m,n)}_{q+1}$ is precisely
    \[
    \left\{
    \chi^{(km,kn)}_{q+1}
    \,:\,
    k\in(\Z/(q-1)\Z)^\times
    \right\}.
    \]

    Since $\chi^{(m,n)}_{q+1}=\chi^{(n,m)}_{q+1}$, two characters $\chi^{(m,n)}_{q+1}$ and $\chi^{(m',n')}_{q+1}$ are Galois conjugate if and only if there exists $k\in(\Z/(q-1)\Z)^\times$ such that
    \[
    (m',n')\equiv(km,kn)\pmod{q-1},
    \]
    up to permutation of the entries. Hence, the Galois conjugacy classes are in bijection with the orbits of
    \[
    U=(\Z/(q-1)\Z)^\times
    \]
    acting on the set $X$ of unordered pairs $(m,n)$ with $0\le n<m\le q-2$.

    We now count the number of orbits using Burnside's lemma. For each $k\in U$, let $\operatorname{Fix}(k)$ denote the number of unordered pairs $(m,n)\in X$ satisfying
    \[
    (km,kn)\equiv(m,n)\pmod{q-1}.
    \]
    This occurs either when both $m$ and $n$ are fixed individually, or when they are interchanged.

    Put $d=q-1$. The congruence
    \[
    km\equiv m\pmod d
    \]
    is equivalent to
    \[
    (k-1)m\equiv0\pmod d.
    \]
    Hence, the number of solutions $m\in\{0,1,\dots,d-1\}$ is $\gcd(k-1,d)$. Therefore, the number of unordered pairs $(m,n)$ with $m\neq n$ and both fixed is
    \[
    \binom{\gcd(k-1,d)}{2}.
    \]

    Next, suppose
    \[
    km\equiv n\pmod d
    \quad \text{and} \quad
    kn\equiv m\pmod d,
    \]
    with $m\neq n$. Then
    \[
    k^2m\equiv m\pmod d,
    \]
    so
    \[
    (k^2-1)m\equiv0\pmod d.
    \]
    The number of solutions is $\gcd(k^2-1,d)$. Among these, the solutions satisfying
    \[
    km\equiv m\pmod d
    \]
    correspond to the case $m=n$, and there are $\gcd(k-1,d)$ such solutions. Since each unordered pair is counted twice, the number of unordered pairs interchanged by $k$ is
    \[
    \frac{1}{2}\bigl(\gcd(k^2-1,d)-\gcd(k-1,d)\bigr).
    \]

    Consequently, we get
    \[
    \operatorname{Fix}(k)
    =
    \binom{\gcd(k-1,d)}{2}
    +
    \frac{1}{2}
    \bigl(
    \gcd(k^2-1,d)-\gcd(k-1,d)
    \bigr).
    \]

    By Burnside's lemma, the number of distinct orbits is
    \[
    \frac{1}{|U|}
    \sum_{k\in U}\operatorname{Fix}(k).
    \]
    Since $|U|=\phi(d)=\phi(q-1)$, we obtain the required formula. Moreover, the quantities $\gcd(k-1,d)$ and $\gcd(k^2-1,d)$ depend only on the prime-power factorisation of $d$, and hence the result depends only on the prime factorisation of $q-1$.

    \item Let $\chi=\chi^{(n)}_{q-1}$. By Lemma~\ref{lemma:CharacterField}, we have
    \[
    \Q(\chi)
    =
    \Q(\chi^{(n)}_{q-1})
    =
    \Q(\eta^n+\eta^{nq}),
    \]
    where $\eta$ is a primitive $(q^2-1)$-th root of unity in $\mathbb{C}$. Let $d=\frac{q^2-1}{\gcd(n,q^2-1)}$. Then $\eta^n$ is a primitive $d$-th root of unity, say $\zeta_d$. The condition $(q+1)\nmid n$ is equivalent to $d\nmid(q-1)$.

    Since $q^2\equiv1\pmod d$ and $q\not\equiv1\pmod d$, the map $\zeta_d\mapsto\zeta_d^q$ defines a nontrivial involution. Its fixed field is $\Q(\zeta_d+\zeta_d^q)$. All values of $\chi$ lie in this field, and
    \[
    \chi^{(n)}_{q-1}(B_1(1))
    =
    -(\zeta_d+\zeta_d^q)
    \]
    generates it. Hence, we get $\Q(\chi)=\Q(\zeta_d+\zeta_d^q)$. The group $(\mathbb{Z}/(q^2-1)\mathbb{Z})^\times$ acts transitively on the primitive $d$-th roots of unity, so all characters corresponding to the same value of $d$ are Galois conjugate. Different values of $d$ produce different fields. Therefore, for each divisor $d$ of $q^2-1$ satisfying $d\nmid(q-1)$, there is exactly one Galois conjugacy class having a representative $\chi$ such that
    \[
    \Q(\chi)=\Q(\zeta_d+\zeta_d^q).
    \]
    The number of such divisors is $\tau(q^2-1)-\tau(q-1)$. Hence, the result follows.
    \qedhere
\end{enumerate}
\end{proof}

The notion of the Schur index was introduced by Schur in 1906. Let $G$ be a finite group and let $\chi$ be an absolutely irreducible character of $G$. Let $\mathbb{F}$ be a field, and let $\psi$ denote the sum of the Galois conjugates of $\chi$ over $\mathbb{F}$. Then $\psi$ takes values in $\mathbb{F}$. The smallest integer $m>0$ such that $m\psi$ is afforded by an $\mathbb{F}$-representation of $G$ is called the \emph{Schur index} of $\chi$ over $\mathbb{F}$, and is denoted by $m_{\mathbb{F}}(\chi)$.

An equivalent interpretation of the Schur index is given in terms of representation fields. There exists a field $\mathbb{E}$ containing $\mathbb{F}(\chi)$ such that $G$ admits an $\mathbb{E}$-representation affording the character $\chi$, and
\[
[\mathbb{E}:\mathbb{F}(\chi)]
=
m_{\mathbb{F}}(\chi).
\]
Moreover, $m_{\mathbb{F}}(\chi)$ is the minimum possible degree among all such field extensions.

Gow~\cite{Gow} computed the Schur index $m_\Q(\chi)$ over the rational field for irreducible complex characters $\chi$ of $\GL_n(q)$, and proved that $m_\Q(\chi)=1$ for $n\le4$. Subsequently, Ohmori~\cite{Ohmori} extended Gow's result to arbitrary $n$ when $q$ is a power of an odd prime.

\begin{lemma}\cite[Theorem~4]{Gow}\label{lemma:schurindexGL}
Let $G=\GL_2(q)$, where $q$ is a prime power. Then $m_{\Q}(\chi)=1$ for every $\chi\in \irr(G)$.
\end{lemma}

Let $G$ be a finite group, and let $\chi\in \irr(G)$. By Schur theory, there exists an irreducible $\Q$-representation $\rho$ of $G$ affording the character $\Omega(\chi)$. Conversely, if $\rho$ is an irreducible $\Q$-representation of $G$, then there exists $\chi\in \Irr(G)$ such that $\Omega(\chi)$ is the character afforded by $\rho$. Moreover,
\[
\deg\rho
=
\Omega(\chi)(1)
=
m_{\Q}(\chi)\,[\Q(\chi):\Q]\,\chi(1)
\]
(see Lemma~\ref{SC}). Hence, distinct Galois conjugacy classes correspond to distinct irreducible rational representations of $G$.

We are now ready to prove Theorem~\ref{thm:RationRepsGL2}.

\begin{proof}[Proof of Theorem~\ref{thm:RationRepsGL2}]
Let $G=\GL_2(q)$, where $q$ is a prime power. Observe that the distinct Galois conjugacy classes of $\irr(G)$ over $\Q$ correspond bijectively to the distinct irreducible rational representations of $G$. Let $\chi\in \irr(G)$ be a representative of a Galois conjugacy class over $\Q$, and let $\rho$ be the corresponding irreducible $\Q$-representation of $G$ affording the character $\Omega(\chi)$. Then
\[
\chi_\rho(1)
=
\Omega(\chi)(1)
=
m_{\Q}(\chi)\,[\Q(\chi):\Q]\,\chi(1).
\]
By Lemma~\ref{lemma:schurindexGL}, we have $m_{\Q}(\chi)=1$ for all $\chi\in\irr(G)$. We now prove each part separately.

\begin{enumerate}
    \item By Lemma~\ref{lemma:galois}, for each divisor $d$ of $q-1$, there is exactly one Galois conjugacy class having a representative $\chi$ with $\chi(1)=1$ and $\Q(\chi)=\Q(\zeta_d)$. Hence, for each divisor $d\mid(q-1)$, there exists exactly one simple $\Q G$-module, up to isomorphism, of dimension $\varphi(d)$
    affording the character $\Omega(\chi^{(n)}_1)$, where $d=\frac{q-1}{\gcd(n,q-1)}$.

    \item Again by Lemma~\ref{lemma:galois}, for each divisor $d$ of $q-1$, there is exactly one Galois conjugacy class having a representative $\chi$ with $\chi(1)=q$ and $\Q(\chi)=\Q(\zeta_d)$. Therefore, for each divisor $d\mid(q-1)$, there exists exactly one simple $\Q G$-module, up to isomorphism, of dimension $q\varphi(d)$ affording the character $\Omega(\chi^{(n)}_q)$, where $d=\frac{q-1}{\gcd(n,q-1)}$.

    \item By Lemma~\ref{lemma:galois}, for each divisor $d$ of $q^2-1$ satisfying $d\nmid(q-1)$, there is exactly one Galois conjugacy class having a representative $\chi$ with
    \[
    \chi(1)=q-1
    \quad \text{and} \quad
    \Q(\chi)=\Q(\zeta_d+\zeta_d^{\,q}).
    \]
    Note that $[\Q(\zeta_d+\zeta_d^{\,q}) : \Q]=\frac{1}{2}[\Q(\zeta_d) : \Q]$
    Hence, for each divisor $d\mid(q^2-1)$ with $d\nmid(q-1)$, there exists exactly one simple $\Q G$-module, up to isomorphism, of dimension
    \[
    \frac{q-1}{2}\,\varphi(d)
    \]
    affording the character $\Omega(\chi^{(n)}_{q-1})$, where $d=\frac{q^2-1}{\gcd(n,q^2-1)}$.

    \item Let
    \[
    U=(\mathbb{Z}/(q-1)\mathbb{Z})^\times
    \]
    act on the set $X$ of unordered pairs $(m,n)$ with $0 \le m, n \le q-2$ and $m \neq n$ by multiplication modulo $q-1$, and let $\mathcal{O}$ be a set of representatives for the orbits of this action. By Lemma~\ref{lemma:galois}, for each $(m,n)\in\mathcal{O}$, there is exactly one Galois conjugacy class having a representative
    \[
    \chi=\chi^{(m,n)}_{q+1}
    \]
    such that
    \[
    \Q(\chi)
    =
    \Q(\zeta_{q-1}^{\,m+n},\,\zeta_{q-1}^{\,m}+\zeta_{q-1}^{\,n}),
    \]
    where $\zeta_{q-1}$ is a primitive $(q-1)$-st root of unity. Therefore, for each $(m,n)\in\mathcal{O}$, there exists exactly one simple $\Q G$-module, up to isomorphism, of dimension
    \[
    (q+1)\,[\Q(\zeta_{q-1}^{\,m+n},\,\zeta_{q-1}^{\,m}+\zeta_{q-1}^{\,n}):\Q] 
    \]
    affording the character $\Omega(\chi^{(m,n)}_{q+1})$.
It remains to express this dimension in terms of the orbit size. By Lemma~\ref{lemma:galois}, the stabilizer of $\mathbb{Q}(\chi)$ in
\[
\operatorname{Gal}(\mathbb{Q}(\zeta_{q-1})/\mathbb{Q})
\cong U
\]
coincides with the stabilizer of the unordered pair $(m,n)$ under the action of $U$. Hence, by the Galois correspondence and the orbit--stabilizer theorem, we obtain
\[
[\mathbb{Q}(\chi):\mathbb{Q}]
=\frac{|U|}{|\operatorname{Stab}_U(m,n)|}
=|\mathcal{O}_{(m,n)}|.
\]
Therefore, we have
\[
(q+1)\,[\mathbb{Q}(\zeta_{q-1}^{\,m+n},
\zeta_{q-1}^{\,m}+\zeta_{q-1}^{\,n}):\mathbb{Q}]
=(q+1)[\mathbb{Q}(\chi):\mathbb{Q}]
=(q+1)|\mathcal{O}_{(m,n)}|.
\]
This completes the proof.

    \item Finally, observe that the counting carried out in the preceding four parts of the theorem exhausts all irreducible rational representations of $G$. Therefore, from Lemma~\ref{lemma:galois}, we obtain
    \begin{align*}
    |\irr_{\Q}(G)|
    &=
    \tau(q-1)
    +
    \tau(q-1)
    +
    \bigl(\tau(q^2-1)-\tau(q-1)\bigr)
    +
    |\mathcal{O}| \\
    &=
    \tau(q-1)+\tau(q^2-1)+|\mathcal{O}|.
    \end{align*}
\end{enumerate}

This completes the proof of Theorem~\ref{thm:RationRepsGL2}.
\end{proof}

\section{Rational character table of \(\mathrm{GL}_2(q)\)}\label{sec:RatCharTab}

In this section, we briefly introduce the rational character table of $G=\mathrm{GL}_2(q)$. We retain the notation of Section~\ref{sec:conj-char}. Thus, throughout this section $\varepsilon$ is a primitive $(q-1)$-th root of unity and $\eta$ is a primitive $(q^{2}-1)$-th root of unity.

Let $C_G(x)$ denote the ordinary conjugacy class of $x\in G$. For a finite-dimensional representation $\rho$ of $G$ over a field $\mathbb F$ of characteristic $0$, let $\chi_\rho$ denote its character. Two elements $x,y\in G$ are called $\mathbb F$-conjugate if $\chi_\rho(x)=\chi_\rho(y)$ for every such $\rho$. Write $C_{\mathbb F}(x)$ for the $\mathbb F$-conjugacy class of $x$. For $\mathbb F=\mathbb Q$, one has
\[
C_{\mathbb Q}(x)=\bigcup_{\substack{m\geq 1\\(m,|x|)=1}} C_G(x^m),
\]
where $|x|$ denotes the order of $x$ in $G$. Furthermore, by the Witt–Berman theorem, the number of $\mathbb Q$-conjugacy classes equals the number of irreducible $\mathbb Q$-representations of $G$ (see \cite[Theorem~42.8]{Curtis-Reiner}).

By Lemma~\ref{lemma:schurindexGL}, every irreducible complex character of $G$ has Schur index $1$ over $\mathbb Q$. Hence, the irreducible rational character attached to $\chi\in\operatorname{Irr}(G)$ is
\[
\Omega(\chi)=\sum_{\sigma\in \operatorname{Gal}(\mathbb Q(\chi)/\mathbb Q)}\chi^{\sigma}.
\]
From now onward, we also call the $\mathbb Q$-conjugacy class of an element $x\in G$ the rational conjugacy class of $x$. The rational conjugacy classes of $G=\mathrm{GL}_2(q)$ are indexed as follows. For $A_1(a)$ and $A_2(a)$ (defined in Section~\ref{sec:conj-char}), the rational conjugacy classes are indexed by $d\mid q-1$, where
\[
d=\frac{q-1}{\gcd(a,q-1)}.
\]
Let $U=(\mathbb{Z}/(q-1)\mathbb{Z})^\times$ act on the set $X$ of unordered pairs $(m,n)$ with $0\le m,n\le q-2$ and $m\neq n$ by multiplication modulo $q-1$. For $A_3(a,b)$ (defined in Section~\ref{sec:conj-char}), the rational conjugacy classes are indexed by a set of representatives $\mathcal O$ of the orbits of the action of the group $U$ on $X$. Futher, for $B_1(a)$ (defined in Section~\ref{sec:conj-char}), the rational conjugacy classes are indexed by $d\mid q^2-1$ with $d\nmid q-1$, where
\[
d=\frac{q^2-1}{\gcd(a,q^2-1)}.
\]

We now define the Ramanujan sum for $d\mid q-1$ as
\[
c_d(r)=\sum_{\substack{1\le t\le d\\(t,d)=1}}\zeta_d^{\,t r}.
\]
For a representative $(m,n)\in\mathcal O$, let $\mathcal O_{(m,n)}$ denote the corresponding orbit under the action of $U$ on $X$ and define
\[
P_{\mathcal O_{(m,n)}}(a)=\sum_{(m',n')\in\mathcal O_{(m,n)}}\varepsilon^{(m'+n')a}
\quad\text{and}\quad
Q_{\mathcal O_{(m,n)}}(a,b)=\sum_{(m',n')\in\mathcal O_{(m,n)}}
\bigl(\varepsilon^{m'a+n'b}+\varepsilon^{n'a+m'b}\bigr).
\]
For $d\mid q^2-1$ with $d\nmid q-1$, the Frobenius automorphism $\sigma_q:\zeta_d\mapsto\zeta_d^{\,q}$ has order $2$. To avoid double-counting in the Galois sum, let $R_d\subset(\mathbb Z/d\mathbb Z)^\times$ be a set of representatives for the quotient $(\mathbb Z/d\mathbb Z)^\times/\langle q\rangle$. For classes of cuspidal characters where the character value is a single power $\zeta_d^{\,tr}$ (as in $A_1,A_2,A_3$ of the fourth row in the rational character table of $G$, set
\[
u_d(r)=\sum_{t\in R_d}\zeta_d^{\,t r}.
\]
For the $B_1(a)$ class, where the value already contains both Frobenius conjugates, set
\[
v_d(r)=\sum_{t\in R_d}\bigl(\zeta_d^{\,t r}+\zeta_d^{\,t q r}\bigr).
\]

\begin{remark}
By elementary number theory, the rational character values of $G$ can, in general, be determined directly without explicitly evaluating the corresponding sums of roots of unity. In particular, the Ramanujan sum can be evaluated as
\[
c_d(r)
=
\mu\!\left(\frac{d}{\gcd(d,r)}\right)
\frac{\varphi(d)}
{\varphi\!\left(d/\gcd(d,r)\right)},
\]
where $\mu$ denotes the Möbius function 
\[
\mu(n)=
\begin{cases}
1, & \text{if }n=1,\\
0, & \text{if }p^2\mid n\text{ for some prime }p,\\
(-1)^k, & \text{if }n\text{ is a product of }k\text{ distinct primes}.
\end{cases}
\]
Furthermore, we have $v_d(r)=c_d(r)$. Thus, both $c_d(r)$ and $v_d(r)$ can be determined directly using the Möbius function and Euler's totient function, without explicitly evaluating the corresponding sums of roots of unity.
\end{remark}

The irreducible rational characters of $G=\GL_2(q)$ are given in Table~\ref{rational table-gl(2)}. In the first two rows, $d\mid q-1$, i.e., $d$ is a positive integer divisor of $q-1$. In the third row, $\mathcal O_{(m,n)}$ runs over the $U$-orbits on $X$; equivalently, the third row is indexed by $(m,n)\in\mathcal O$, where $\mathcal O$ is a set of representatives for the orbits of the action of $U$ on $X$, and $\mathcal O_{(m,n)}$ denotes the orbit corresponding to $(m,n)$. In the fourth row, $d\mid q^2-1$ with $d\nmid q-1$, i.e., $d$ is a positive integer divisor of $q^2-1$ that does not divide $q-1$. All entries are rational integers. The Witt--Berman theorem ensures the completeness of the table.

\begin{table}[h]
\centering
\caption{The rational character table of $G=\GL_2(q)$}\label{rational table-gl(2)}
\[
\begin{array}{c|c|cccc}
\text{Character} & \deg & A_1(a) & A_2(a) & A_3(a,b) & B_1(a)\\ \hline
\Omega(\chi_1^{(d)}) & \varphi(d) & c_d(2a) & c_d(2a) & c_d(a+b) & c_d(a)\\[1.2ex]
\Omega(\chi_q^{(d)}) & q\varphi(d) & q\,c_d(2a) & 0 & c_d(a+b) & -c_d(a)\\[1.2ex]
\Omega(\chi_{q+1}^{(m,n)}) & (q+1)|\mathcal O_{(m,n)}| & (q+1)P_{\mathcal O_{(m,n)}}(a) & P_{\mathcal O_{(m,n)}}(a) & Q_{\mathcal O_{(m,n)}}(a,b) & 0\\[1.2ex]
\Omega(\chi_{q-1}^{(d)}) & \dfrac{q-1}{2}\,\varphi(d) & (q-1)u_d(a(q+1)) & -u_d(a(q+1)) & 0 & -v_d(a)
\end{array}
\]
\end{table}

\section{Matrix representations of $\GL_2(q)$ over $\Q$}\label{sec:MatrixReps}
Let $G=\mathrm{GL}_2(q)$ and let $\chi\in\Irr(G)$. Suppose $\rho:G\to \mathrm{GL}_d(\mathbb{C})$ is an irreducible complex representation of $G$ affording the character $\chi$. Set $\mathbb{L}=\mathbb{Q}(\chi)$ and fix a $\mathbb{Q}$-basis $\{b_1,\dots,b_t\}$ of $\mathbb{L}$, where $t=[\mathbb{L}:\mathbb{Q}]$. By Lemma~\ref{lemma:schurindexGL}, we have $m_{\mathbb{Q}}(\chi)=1$. Therefore, there exists a conjugate representation $\rho' : G \to \mathrm{GL}_d(\mathbb{L})$ of $\rho$. For each $g\in G$, write $\rho'(g) = (a_{ij}(g))$, and hence $a_{ij}(g)\in\mathbb{L}$ for all $i$ and $j$.

For each entry $a_{ij}(g)$, let $\M(a_{ij}(g))$ denote the $t\times t$ matrix over $\mathbb{Q}$ corresponding to multiplication by $a_{ij}(g)$ on $\mathbb{L}$ with respect to the basis $\{b_j\}$. Define
\[
\widetilde{\rho}(g)=\bigl(\M(a_{ij}(g))\bigr)_{i,j=1}^{d}.
\]
Then $\widetilde{\rho}(g)$ is a block matrix of size $dt$ with entries in $\mathbb{Q}$. Moreover, $\widetilde{\rho}:G\to \mathrm{GL}_{dt}(\mathbb{Q})$ is a representation satisfying
\[
\operatorname{Tr}_{\mathbb{Q}}\bigl(\widetilde{\rho}(g)\bigr)
=
\sum_{\sigma\in\Gal(\mathbb{L}/\mathbb{Q})}
\sigma\bigl(\operatorname{Tr}_{\mathbb{L}}(\rho(g))\bigr)
=
\sum_{\sigma\in\Gal(\mathbb{L}/\mathbb{Q})}\chi^\sigma(g)
=
\Omega(\chi)(g)
\]
for all $g\in G$. Consequently, $\widetilde{\rho}$ defines an irreducible $\mathbb{Q}$-representation of $G$ affording the character $\Omega(\chi)$. However, there is no general systematic procedure for constructing such a $\rho'$.

In this section, we present an alternative approach for constructing irreducible matrix representations of $\GL_2(q)$. Recall that, for a finite group $G$ and $\chi\in\irr(G)$, there exists a unique irreducible $\mathbb{Q}$-representation $\rho$ of $G$ such that $\chi$ occurs as an irreducible constituent of $\rho\otimes_{\mathbb{Q}}\mathbb{F}$ with multiplicity $m_{\mathbb{Q}}(\chi)$, where $\mathbb{F}$ is a splitting field for $G$. In the case of a linear complex character $\psi$ of $G$, an irreducible matrix representation of $G$ over $\mathbb{Q}$ affording the character $\Omega(\psi)$ can be constructed from Lemma~\ref{lemma:YamadaLinear}.

\begin{lemma} \cite[Proposition 1]{Y}\label{lemma:YamadaLinear}
Let $G$ be a finite group, let $\psi\in\lin(G)$, and set $N=\ker(\psi)$ with $n=[G:N]$. Suppose that
\[
G=\bigcup_{i=0}^{n-1}Ny^i.
\]
Then $\psi(xy^i)=\zeta_n^i$, where $x\in N$, $0\leq i<n$, and $\zeta_n$ is a primitive $n$-th root of unity. Let
\[
f(X)=X^s-a_{s-1}X^{s-1}-\cdots-a_1X-a_0
\]
be the irreducible polynomial over $\mathbb{Q}$ satisfying $f(\zeta_n)=0$, where $s=\phi(n)$. Define
\[
\Psi(xy^i)=
\left(
\begin{array}{ccccc}
0 & 1 & 0 & \cdots & 0\\
0 & 0 & 1 & \cdots & 0\\
\vdots & \vdots & \vdots & \ddots & \vdots\\
0 & 0 & \cdots & 0 & 1\\
a_0 & a_1 & \cdots & \cdots & a_{s-1}
\end{array}
\right)^i,
\qquad (0\leq i<n,\ x\in N).
\]
Then $\Psi$ is an irreducible $\mathbb{Q}$-representation of $G$ affording the character $\Omega(\psi)$.
\end{lemma}

Let $G$ be a finite group and $H$ a subgroup of $G$. The following lemma describes the relation between $\Omega(\psi)$ and $\Omega(\psi^G)$ for $\psi\in\Irr(H)$.

\begin{lemma} \cite[Proposition 3]{Y}\label{lemma:Yamada}
Let $G$ be a finite group, let $H$ be a subgroup of $G$, and let $\psi\in\irr(H)$ such that $\psi^G\in\Irr(G)$. Then $m_\mathbb{Q}(\psi^G)$ divides $ m_\mathbb{Q}(\psi)[\mathbb{Q}(\psi):\mathbb{Q}(\psi^G)]$.
Moreover, the induced character $\Omega(\psi)^G$ is the character of an irreducible $\mathbb{Q}$-representation of $G$ if and only if
\[
m_\mathbb{Q}(\psi^G)
=
m_\mathbb{Q}(\psi)[\mathbb{Q}(\psi):\mathbb{Q}(\psi^G)].
\]
In this case, we have $\Omega(\psi)^G=\Omega(\psi^G)$.
\end{lemma}

Let $G=\mathrm{GL}_2(q)$. If $\chi\in\Irr(G)$ with $\chi(1)=1$, then an irreducible rational matrix representation of $G$ affording the character $\Omega(\chi)$ can be constructed using Lemma~\ref{lemma:YamadaLinear}. We now describe a similar construction of irreducible matrix representations of $\GL_2(q)$ over $\Q$ affording the character $\Omega(\chi)$, where $\chi\in\Irr(G)$ is any of the remaining irreducible complex characters of $\GL_2(q)$ obtained via parabolic induction.

\subsection{The $q$-dimensional irreducible complex representations of $\mathrm{GL}_2(q)$}

Let $\mathbb{F}_q$ be the finite field with $q$ elements, and let $G=\mathrm{GL}_2(q)$. The \emph{projective line} $\mathbb{P}^1(\mathbb{F}_q)$ is the set of $1$-dimensional subspaces of $\mathbb{F}_q^2$, consisting of $q+1$ points. In homogeneous coordinates, $[x:y]$ with $(x,y)\neq (0,0)$ and $[x:y]=[\lambda x:\lambda y]$ for $\lambda\in\mathbb{F}_q^\times$. The group $G$ acts on $\mathbb{P}^1(\mathbb{F}_q)$ via
\[
\begin{pmatrix} a & b \\ c & d \end{pmatrix}\cdot [x:y]
=
[ax+by:cx+dy].
\]

The permutation representation $\pi$ on
$
V=\{f:\mathbb{P}^1(\mathbb{F}_q)\to\mathbb{C}\},
$
defined by $(\pi(g)f)(P)=f(g^{-1}\cdot P)$, decomposes as
$
\mathbf{1}\oplus \operatorname{St},
$
where $\mathbf{1}$ is the trivial representation and
\[
\operatorname{St}
=
\left\{
f\in V:\sum_{P\in\mathbb{P}^1(\mathbb{F}_q)}f(P)=0
\right\}
\]
is the $q$-dimensional \emph{Steinberg representation} of $G$ (see \cite{Steinberg}).

For any multiplicative character $\mu:\mathbb{F}_q^\times\to\mathbb{C}^\times$, the composition
$
\mu\circ\det:G\to\mathbb{C}^\times
$
given by
$
(\mu\circ\det)(g)=\mu(\det g)
$
defines a one-dimensional irreducible character of $G$. Twisting the Steinberg representation by this character yields
$
\operatorname{St}\otimes(\mu\circ\det),
$
which is again irreducible of degree $q$. As $\mu$ ranges over the $q-1$ distinct characters of $\mathbb{F}_q^\times$, these representations are pairwise inequivalent and exhaust all irreducible complex representations of $\mathrm{GL}_2(q)$ of degree $q$.

Observe that the Steinberg representation $\operatorname{St}$ of $\GL_2(q)$ is realizable over $\Q$. Since $(\mu\circ\det)\in\lin(\GL_2(q))$, Lemma~\ref{lemma:YamadaLinear} provides an explicit irreducible rational matrix representation affording the character $\Omega(\mu\circ\det)$. Let $\rho_\mu$ be an irreducible $\Q$-representation of $\GL_2(q)$ affording $\Omega(\mu\circ\det)$. Let $\chi_q^{(n)}$ be an irreducible complex character of $\GL_2(q)$ of degree $q$. Then there exists a multiplicative character $\mu:\mathbb{F}_q^\times\to\mathbb{C}^\times$ such that $\chi_q^{(n)}$ is afforded by $\operatorname{St}\otimes(\mu\circ\det)$. We now prove Theorem~\ref{thm:Steinberg}, showing that $\operatorname{St}\otimes\rho_\mu$ is an irreducible rational matrix representation of $\GL_2(q)$ affording the character $\Omega(\chi_q^{(n)})$.

\begin{theorem}\label{thm:Steinberg}
Let $G=\GL_2(q)$, where $q$ is a prime power. Let $\chi^{(n)}_q \in \irr(G)$ be a character of degree $q$.
Suppose $\mu:\mathbb{F}_q^\times\to\mathbb{C}^\times$ is the multiplicative character such that $\chi^{(n)}_q$ is afforded by $\operatorname{St}\otimes (\mu\circ\det)$.
Let $\rho_\mu$ be an irreducible $\Q$-representation of $G$ affording the character $\Omega(\mu\circ\det)$.
Then $\operatorname{St}\otimes \rho_\mu$ is an irreducible rational representation of $G$ affording the character $\Omega(\chi^{(n)}_q)$.
\end{theorem}

\begin{proof}
For any $\sigma\in\Gal(\Q(\mu)/\Q)$, we have
$
(\mu\circ\det)^\sigma=\mu^\sigma\circ\det.
$
Let $\chi_{\operatorname{St}}$ denote the character of the Steinberg representation $\operatorname{St}$ of $G$. Since $\operatorname{St}$ is realizable over $\Q$, we have $\chi_{\operatorname{St}}(g)\in\mathbb{Z}$ for every $g\in G$. Hence, $\chi_{\operatorname{St}}$ is fixed by each $\sigma\in\Gal(\Q(\mu)/\Q)$.

By hypothesis,
$
\chi_q^{(n)}=\chi_{\operatorname{St}}\cdot(\mu\circ\det).
$
Let $\rho_\mu$ be an irreducible $\Q$-representation of $G$ affording the character $\Omega(\mu\circ\det)$. Then the character of $\operatorname{St}\otimes\rho_\mu$ is
\[
\chi_{\operatorname{St}\otimes\rho_\mu}
=
\chi_{\operatorname{St}}\cdot\chi_{\rho_\mu}
=
\chi_{\operatorname{St}}\cdot\Omega(\mu\circ\det)
=
\chi_{\operatorname{St}}\cdot
\sum_{\sigma\in\Gal(\Q(\mu)/\Q)}
(\mu^\sigma\circ\det).
\]
Since $\chi_{\operatorname{St}}$ is $\sigma$-invariant,
\[
\chi_{\operatorname{St}}\cdot(\mu^\sigma\circ\det)
=
(\chi_{\operatorname{St}}\cdot(\mu\circ\det))^\sigma
=
(\chi_q^{(n)})^\sigma.
\]
Therefore, we have
\[
\chi_{\operatorname{St}\otimes\rho_\mu}
=
\sum_{\sigma}(\chi_q^{(n)})^\sigma
=
\Omega(\chi_q^{(n)}).
\]

It remains to show that $\operatorname{St}\otimes\rho_\mu$ is irreducible over $\Q$. Since $\operatorname{St}$ is absolutely irreducible over $\Q$, the Galois conjugates of
$
\chi_q^{(n)}
=
\chi_{\operatorname{St}}\cdot(\mu\circ\det)
$
are precisely
$
\chi_{\operatorname{St}}\cdot(\mu^\sigma\circ\det).
$
Hence, $\Q(\chi_q^{(n)})=\Q(\mu)$. Consequently, $\operatorname{St}\otimes\rho_\mu$ affords the character $\Omega(\chi_q^{(n)})$ and is therefore irreducible over $\Q$. This completes the proof of Theorem~\ref{thm:Steinberg}.
\end{proof}

Let $G=\GL_2(q)$, where $q$ is a prime power, and let $\chi, \psi \in \nl(G)$. Then $\chi \otimes \psi$ is not irreducible. However, this statement does not necessarily hold for irreducible rational characters of $G$. Theorem~\ref{thm:Steinberg} yields the following corollary.

\begin{corollary}\label{cor:tensor}
    Let $G=\GL_2(q)$, where $q$ is a prime power. Then there exist irreducible rational representations $\rho$ and $\theta$ of $G$, each of degree greater than one, such that $\rho \otimes_{\mathbb{Q}} \theta$ is also an irreducible rational representation of $G$.
\end{corollary}

\subsection{The $(q+1)$-dimensional irreducible complex representations of $\mathrm{GL}_2(q)$}
Let $\mathbb{F}_q$ be the finite field with $q$ elements, and let $G=\mathrm{GL}_2(q)$. Let $B$ denote the Borel subgroup of upper triangular matrices,
\[
B=\left\{ \begin{pmatrix} a & b \\ 0 & d \end{pmatrix} : a,d\in\mathbb{F}_q^\times,\; b\in\mathbb{F}_q \right\},
\]
and let $T\cong \mathbb{F}_q^\times\times\mathbb{F}_q^\times$ be the diagonal torus. For multiplicative characters $\mu,\nu:\mathbb{F}_q^\times\to\mathbb{C}^\times$, define a one-dimensional representation of $B$ by
\[
(\mu\otimes\nu)\begin{pmatrix} a & b \\ 0 & d \end{pmatrix}=\mu(a)\nu(d),
\]
and consider the induced representation $\pi_{\mu,\nu}=\operatorname{Ind}_B^G(\mu\otimes\nu)$. Its dimension is $[G:B]=q+1$. If $\mu\neq\nu$, then $\pi_{\mu,\nu}$ is irreducible. Moreover, the character of $\pi_{\mu,\nu}$ is $\chi^{(m,n)}_{q+1}$ in the character table of $G$, where $\mu(\sigma^a)=\varepsilon^{ma}$ and $\nu(\sigma^a)=\varepsilon^{na}$ for a generator $\sigma$ of $\mathbb{F}_q^\times$ and a primitive $(q-1)$-th root of unity $\varepsilon$ (see Table~\ref{table-gl(2)}).

To explicitly construct an irreducible matrix representation of $G$ affording the character $\Omega(\chi^{(m,n)}_{q+1})$ over $\Q$, one starts with the Borel subgroup. Since $\mu\otimes\nu$ is a one-dimensional character of $B$, Lemma~\ref{lemma:YamadaLinear} yields an explicit irreducible $\Q$-representation $\rho_B$ of $B$ with character
\[
\Omega(\mu\otimes\nu)=\sum_{\tau\in\operatorname{Gal}(\mathbb{Q}(\mu\otimes\nu)/\mathbb{Q})}(\mu\otimes\nu)^\tau.
\]
Inducing $\rho_B$ to $G$ gives a rational representation $\Pi_{\mu,\nu}=\operatorname{Ind}_B^G(\rho_B)$ whose character is
\[
\chi_{\Pi_{\mu,\nu}}=(\Omega(\mu\otimes\nu))^G.
\]
Theorem~\ref{thm:principal} gives a necessary and sufficient condition for $\Pi_{\mu,\nu}$ to be irreducible. Furthermore, Theorem~\ref{thm:principal} describes the outcome in two distinct cases. Before proving Theorem~\ref{thm:principal}, we establish Lemma~\ref{lemma:EqualFields}, which is essential for the proof.

\begin{lemma}\label{lemma:EqualFields}
Let $q\ge 3$ be an integer, $N=q-1$, and let $\varepsilon$ be a primitive $N$-th root of unity.  
For integers $0\le n<m\le N-1$, set $\mathbb{L}=\mathbb{Q}(\varepsilon^m,\varepsilon^n)$ and $\mathbb{K}=\mathbb{Q}(\varepsilon^{m+n},\varepsilon^m+\varepsilon^n)$. Define $d=\gcd(m,n,N)$ and $M=\frac{N}{d}$.
Then
\[
\mathbb{L}=\mathbb{K} \quad\Longleftrightarrow\quad 
\left(\frac{m}{d}\right)^2 \not\equiv \left(\frac{n}{d}\right)^2 \pmod M.
\]
\end{lemma}

\begin{proof}
Set $x=\varepsilon^m$ and $y=\varepsilon^n$.  
Then $\mathbb{L}=\mathbb{Q}(x,y)$ and $\mathbb{K}=\mathbb{Q}(x+y,xy)$. Because $\varepsilon$ has order $N$, the element $\varepsilon^d$ has order $M=N/d$.  
Write
\[
x=\varepsilon^m=(\varepsilon^d)^{m/d}\, \, \, \, \text{and} \, \, \, \,
y=\varepsilon^n=(\varepsilon^d)^{n/d}
\]
with $\gcd(\frac{m}{d}, \frac{n}{d},M)=1$.  
Bezout's identity gives integers $a,b,c$ such that
\[
a\frac{m}{d}+b\frac{n}{d}+cM=1.
\]
Hence, we get
\[
\varepsilon^d = (\varepsilon^m)^a(\varepsilon^n)^b \in \mathbb{L}.
\]
Therefore, $\mathbb{L}=\mathbb{Q}(\varepsilon^d)=\mathbb{Q}(\zeta_M)$, where $\zeta_M=\varepsilon^d$ is a primitive $M$-th root of unity.

Further, note that both $x$ and $y$ satisfy $T^2-(x+y)T+xy=0$, a polynomial with coefficients in $K$. Therefore, we have $[\mathbb{L} : \mathbb{K}]\le 2$. Thus, either $\mathbb{L} = \mathbb{K}$ or $[\mathbb{L} : \mathbb{K}] = 2$.

Assume $[\mathbb{L} : \mathbb{K}] = 2$. Then there exists a non‑trivial automorphism $\sigma\in\operatorname{Gal}(\mathbb{L} / \mathbb{K})$ of order $2$. Since $\sigma$ fixes $x+y$ and $xy$, it permutes the two roots $x,y$ of $T^2-(x+y)T+xy$. Because $\sigma$ is not the identity, we have
\[
\sigma(x)=y\, \, \, \, \text{and} \, \, \, \, \sigma(y)=x.
\]
Every automorphism of $\mathbb{Q}(\zeta_M)$ is of the form $\sigma_a(\zeta_M)=\zeta_M^a$ with $a\in(\mathbb{Z}/M\mathbb{Z})^\times$.  
Using $x=\zeta_M^{m/d}$ and $y=\zeta_M^{n/d}$, the swapping condition becomes
\[
\sigma_a(x)=y \;\Longleftrightarrow\; a\frac{m}{d}\equiv\frac{n}{d}\pmod M,
\]
\[
\sigma_a(y)=x \;\Longleftrightarrow\; a\frac{n}{d}\equiv\frac{m}{d}\pmod M.
\]
Multiplying the first congruence by $n/d$ and the second by $m/d$ yields
\[
a\frac{mn}{d^2}\equiv\Bigl(\frac{n}{d}\Bigr)^2\pmod M \quad \text{and} \quad
a\frac{mn}{d^2}\equiv\Bigl(\frac{m}{d}\Bigr)^2\pmod M.
\]
Hence, we have
\[
\Bigl(\frac{m}{d}\Bigr)^2\equiv\Bigl(\frac{n}{d}\Bigr)^2\pmod M.
\]
  
Conversely, suppose
\[
\Bigl(\frac{m}{d}\Bigr)^2\equiv\Bigl(\frac{n}{d}\Bigr)^2\pmod M.
\]
Because $\gcd(m/d,M)=1$, the residue
\[
a\equiv \frac{n}{d}\Bigl(\frac{m}{d}\Bigr)^{-1}\pmod M
\]
belongs to $(\mathbb{Z}/M\mathbb{Z})^\times$.  
Then clearly $a\frac{m}{d}\equiv\frac{n}{d}\pmod M$.  
Using the assumed congruence,
\[
a\frac{n}{d}\equiv \frac{n}{d}\Bigl(\frac{m}{d}\Bigr)^{-1}\frac{n}{d}
\equiv \frac{m}{d}\pmod M.
\]
Therefore, the automorphism $\sigma_a$ satisfies $\sigma_a(x)=y$ and $\sigma_a(y)=x$.  
It fixes $x+y$ and $xy$, hence fixes $\mathbb{K}$.  
Since $x\neq y$ (because $n<m$ and $0\le n,m\le N-1$), $\sigma_a$ is non‑trivial, so $[\mathbb{L} : \mathbb{K}]=2$ and $\mathbb{L} \neq \mathbb{K}$. This completes the proof of Lemma~\ref{lemma:EqualFields}.
\end{proof}

With the above notation, we are now ready to prove Theorem~\ref{thm:principal}.

\begin{theorem}\label{thm:principal}
Let $G=\mathrm{GL}_2(q)$, where $q$ is a prime power. Let $\mu,\nu:\mathbb{F}_q^\times\to\mathbb{C}^\times$ be multiplicative characters satisfying $\mu(\sigma)=\varepsilon^{m}$ and $\nu(\sigma)=\varepsilon^{n}$ for a generator $\sigma$ of $\mathbb{F}_q^\times$, where $0\le n<m\le q-2$ and $\varepsilon$ is a primitive $(q-1)$-th root of unity. Let $\chi=\chi^{(m,n)}_{q+1}$ be the principal series character afforded by $\pi_{\mu,\nu}=\operatorname{Ind}_B^G(\mu\otimes\nu)$, where $B$ denotes the Borel subgroup of upper triangular matrices. Define $d=\gcd(m,n,q-1)$ and $M=\frac{q-1}{d}$. Let $\rho_B$ be an irreducible $\Q$-representation of $B$ affording the character $\Omega(\mu\otimes\nu)$, and set $\Pi_{\mu,\nu}=\operatorname{Ind}_B^G(\rho_B)$. Then the following hold.
\begin{enumerate}
\item If $\left(\frac{m}{d}\right)^2 \not\equiv \left(\frac{n}{d}\right)^2 \pmod{M}$, then $\Pi_{\mu,\nu}$ is an irreducible $\Q$-representation of $G$ with character $\Omega(\chi)$.

\item If $\left(\frac{m}{d}\right)^2 \equiv \left(\frac{n}{d}\right)^2 \pmod{M}$, then
\[
\Pi_{\mu,\nu}\cong \mathcal{R}\oplus\mathcal{R},
\]
where $\mathcal{R}$ is an irreducible $\Q$-representation of $G$ that affords the character $\Omega(\chi)$. In particular, $\Pi_{\mu,\nu}$ is the direct sum of two isomorphic irreducible $\Q$-representations of $G$ with character $\Omega(\chi)$.
\end{enumerate}
\end{theorem}

\begin{proof}
We prove each part separately.
\begin{enumerate}
    \item Observe that $\Q(\mu\otimes\nu)=\Q(\varepsilon^{m},\varepsilon^{n})$. By Lemma~\ref{lemma:CharacterField},
    \[
    \Q(\chi)=\Q(\varepsilon^{m+n},\varepsilon^m+\varepsilon^n).
    \]
    Since $\left(\frac{m}{d}\right)^2 \not\equiv \left(\frac{n}{d}\right)^2 \pmod{M}$, Lemma~\ref{lemma:EqualFields} yields $\Q(\chi)=\Q(\mu\otimes\nu)$. Moreover, $m_\Q(\chi)=1$ by Lemma~\ref{lemma:schurindexGL}. Therefore, Lemma~\ref{lemma:Yamada} implies that $\Pi_{\mu,\nu}$ is an irreducible $\Q$-representation of $G$ with character $\Omega(\chi)$.

    \item In this case, we have
    \[
    [\Q(\mu\otimes\nu):\Q(\chi)]=2
    \]
    by the proof of Lemma~\ref{lemma:EqualFields}. Since $m_\Q(\chi)=1$ by Lemma~\ref{lemma:schurindexGL}, and induction commutes with Galois automorphisms, we have
    \begin{align*}
       (\mu\otimes\nu)^G=\chi 
       &\implies \sum_{\tau\in\Gal(\Q(\mu\otimes\nu)/\Q)}((\mu\otimes\nu)^G)^\tau
       =\sum_{\tau\in\Gal(\Q(\mu\otimes\nu)/\Q)}\chi^\tau \\
       &\implies (\Omega(\mu\otimes\nu))^G
       =[\Q(\mu\otimes\nu):\Q(\chi)]
       \sum_{\tau\in\Gal(\Q(\chi)/\Q)}\chi^\tau \\
       &\implies \chi_{\Pi_{\mu,\nu}}=2\Omega(\chi).
    \end{align*}
    Thus, $\mathcal{R}$ is an irreducible rational representation of $G$ that affords the character $\Omega(\chi)$. Since $2\Omega(\chi)$ is the character of $\mathcal{R}\oplus\mathcal{R}$, it follows that
    \[
    \Pi_{\mu,\nu}\cong\mathcal{R}\oplus\mathcal{R}.
    \]
\end{enumerate}
This completes the proof of Theorem~\ref{thm:principal}.
\end{proof}

\subsection{An algorithm for the irreducible rational representation
$\mathcal R$ in Theorem~\ref{thm:principal}}

Let $q=p^m$ be a prime power, $G=\mathrm{GL}_2(q)$, and $B$ the upper
triangular Borel. Fix a generator $\sigma$ of $\mathbb F_q^\times$ and a
primitive $(q-1)$-th root of unity $\epsilon$. Let
$\mu,\nu:\mathbb F_q^\times\to\mathbb C^\times$ be multiplicative
characters with
\[
\mu(\sigma)=\epsilon^m \quad \text{and} \quad \nu(\sigma)=\epsilon^n,
\qquad 0\le n<m\le q-2.
\]
Put
\[
d=\gcd(m,n,q-1),\qquad M=\frac{q-1}{d},\qquad
n_0=\operatorname{ord}(\mu\otimes\nu), \quad \text{and} \quad
s=\phi(n_0).
\]
We assume Case~(2) of Theorem~\ref{thm:principal}, i.e.
\[
\Bigl(\frac{m}{d}\Bigr)^2\equiv
\Bigl(\frac{n}{d}\Bigr)^2\pmod M.
\]
Set
\[
N=(q+1)s,\qquad r=\frac{N}{2}
=\frac{(q+1)s}{2},\quad \text{and} \quad
\mathbb K=\mathbb Q(\chi_{q+1}^{(m,n)})
=\mathbb Q(\epsilon^{m+n},\epsilon^m+\epsilon^n).
\]
Let $V_{\Pi_{\mu,\nu}}=V_{\operatorname{Ind}_B^G(\rho_B)}$ be the vector
space underlying $\Pi_{\mu,\nu}=\operatorname{Ind}_B^G(\rho_B)$. This
implies that
\[
\Pi_{\mu,\nu}:G\longrightarrow
\mathrm{GL}(V_{\Pi_{\mu,\nu}})
\quad\text{and}\quad
\dim_{\mathbb Q}V_{\Pi_{\mu,\nu}}=N.
\]
By Theorem~\ref{thm:principal}(2), we have
\[
V_{\Pi_{\mu,\nu}}
\cong V_{\mathcal R}\oplus V_{\mathcal R},
\]
where $V_{\mathcal R}$ is an irreducible $\mathbb QG$-module and $\dim_{\mathbb Q}V_{\mathcal R}=r=\frac{N}{2}$.
Note that $\Pi_{\mu,\nu}$ is defined over $\mathbb Q$: the rational Borel
representation $\rho_B:B\to\mathrm{GL}_s(\mathbb Q)$ is the
companion-matrix representation of Lemma~\ref{lemma:YamadaLinear}, and
$\Pi_{\mu,\nu}$ is induced from it. We denote by
\[
\mathcal R:G\longrightarrow
\mathrm{GL}_{r}(\mathbb Q)
\]
the irreducible rational matrix representation afforded by
$V_{\mathcal R}$ (i.e.\ the homomorphism whose character is
$\Omega(\chi_{q+1}^{(m,n)})$).

Define
\[
\mathcal A=\operatorname{End}_{\mathbb QG}(V_{\Pi_{\mu,\nu}})
=\{X\in M_N(\mathbb Q):
X\Pi_{\mu,\nu}(g)=\Pi_{\mu,\nu}(g)X
\ \text{for all }g\in G\}.
\]
By Theorem~\ref{thm:principal}(2), Schur's lemma, and
$m_{\mathbb Q}(\chi_{q+1}^{(m,n)})=1$ (see
Lemma~\ref{lemma:schurindexGL}), we have
\[
\mathcal A\cong M_2(D),\qquad
D=\operatorname{End}_{\mathbb QG}(V_{\mathcal R})
=\mathbb Q(\chi_{q+1}^{(m,n)})=\mathbb K.
\]
Therefore, $\mathcal A\cong M_2(\mathbb K)$ and $
\dim_{\mathbb Q}\mathcal A=4[\mathbb K:\mathbb Q]$. 

\begin{lemma}\label{lem:idem}
$\mathcal A$ contains an idempotent $E$ with $E^2=E$ and $\operatorname{rank}_{\mathbb Q}(E)
=\dim_{\mathbb Q}V_{\mathcal R}=r$. For any such $E$, the subspace
\[
W_{\mathcal R}:=\operatorname{Im}(E)
\subseteq V_{\Pi_{\mu,\nu}}
\]
is $G$-invariant and isomorphic to $V_{\mathcal R}$ as a
$\mathbb QG$-module.
\end{lemma}

\begin{proof}
Under the isomorphism $\mathcal A\cong M_2(\mathbb K)$, the matrix unit
\[
e=\begin{pmatrix}1&0\\0&0\end{pmatrix}\in M_2(\mathbb K)
\]
corresponds to an idempotent $E\in\mathcal A\subseteq M_N(\mathbb Q)$.
Since $V_{\Pi_{\mu,\nu}}\cong
V_{\mathcal R}\oplus V_{\mathcal R}$
and $E$ acts as the projection onto the first copy along the second, $W_{\mathcal R}=\operatorname{Im}(E)\cong V_{\mathcal R}$ and $\dim_{\mathbb Q}W_{\mathcal R}
=\dim_{\mathbb Q}V_{\mathcal R}=r$.
Conversely, any idempotent $E\in\mathcal A$ of
$\mathbb Q$-rank $r$ has $\operatorname{Im}(E)$ a $G$-invariant
subspace of dimension $r$, hence isomorphic to $V_{\mathcal R}$.
\end{proof}

\begin{remark}
To compute $\mathcal R$, determine $\mathcal A$ and then find
$E\in\mathcal A$ with $E^2=E$ and
$\operatorname{rank}_{\mathbb Q}(E)=r$. Set $W_{\mathcal R}=\operatorname{Im}(E)
\subseteq V_{\Pi_{\mu,\nu}}$. Note that $W_{\mathcal R}$ is a $G$-invariant $\mathbb Q$-subspace
of dimension $r$, and $W_{\mathcal R}\cong V_{\mathcal R}$ by
Lemma~\ref{lem:idem}. Then
\[
\mathcal R=\Pi_{\mu,\nu}|_{W_{\mathcal R}},
\]
the restriction of $\Pi_{\mu,\nu}$ to the invariant subspace
$W_{\mathcal R}$. To construct the matrix representation
$\mathcal R$ explicitly, proceed as follows.

\begin{enumerate}
\item Choose a $\mathbb Q$-basis
$b_1,\dots,b_r$ of $W_{\mathcal R}$
(e.g.\ $r$ independent columns of $E$; recall that $Ee_j$ is the
$j$-th column of $E$, where $e_j$ is the $j$-th standard basis vector).

\item For each generator $g$ of $G$, write
\[
\Pi_{\mu,\nu}(g)b_j
=\sum_{k=1}^{r}c^{(g)}_{kj}b_k,
\qquad c^{(g)}_{kj}\in\mathbb Q,
\]
and set $\mathcal R(g)=\bigl(c^{(g)}_{kj}\bigr)$. Then $\mathcal R:G\longrightarrow\mathrm{GL}_r(\mathbb Q)$
is the irreducible $\mathbb Q$-representation of $G$
corresponding to $V_{\mathcal R}$ and affording the character
$\Omega(\chi_{q+1}^{(m,n)})$.
\end{enumerate}
\end{remark}

\subsection{An example}

We conclude this section with an explicit illustration of the construction
of the matrix representation corresponding to the irreducible rational representation
$\mathcal R$ in Theorem~\ref{thm:principal}.

Let $G=\mathrm{GL}_2(4)$, $m=2$, and $n=1$. Here, $q=4$, hence
$d=1$, $M=3$, and $m^2=4\equiv 1=n^2\pmod 3$. Therefore, Theorem~\ref{thm:principal}(2) applies. Take
$\epsilon=\zeta_3$, a primitive third root of unity in $\mathbb C$. Then
\[
\mathbb K
=\mathbb Q(\epsilon^{m+n},\epsilon^m+\epsilon^n)
=\mathbb Q(\epsilon^3,\epsilon^2+\epsilon)
=\mathbb Q,
\]
since $\epsilon^3=1$ and $\epsilon^2+\epsilon=-1$. Moreover,
$n_0=3$, $s=\phi(3)=2$, and consequently $N=(q+1)s=10$, $r=\frac{N}{2}=5$. Thus, by Theorem~\ref{thm:principal}(2), we have $V_{\Pi_{\mu,\nu}}
\cong V_{\mathcal R}\oplus V_{\mathcal R}$ and $\dim_{\mathbb Q}V_{\mathcal R}=5$.

Write
\[
\mathbb F_4=\{0,1,\alpha,\alpha+1\},
\qquad
\alpha^2+\alpha+1=0,
\]
and define the multiplicative characters $\mu,\nu$ by $\mu(\alpha)=\epsilon^2$ and $\nu(\alpha)=\epsilon$. Let $D=\begin{pmatrix}\alpha&0\\0&1\end{pmatrix}$, $U=\begin{pmatrix}1&1\\0&1\end{pmatrix}$, and $W=\begin{pmatrix}0&1\\1&0\end{pmatrix}$. Then $G=\langle D,U,W\rangle$. For $b=
\begin{pmatrix}
\alpha^u&*\\
0&\alpha^v
\end{pmatrix}\in B$, the irreducible $\mathbb Q$-representation of the Borel subgroup $B$ affording the character $\Omega (\mu \otimes \nu)$ is given by
\[
\rho_B(b)=C^{\,2u+v},
\qquad
C=\begin{pmatrix}0&-1\\1&-1\end{pmatrix},
\]
where the exponent is taken modulo $3$ (see Lemma~\ref{lemma:YamadaLinear}). In particular, $C^3=I_2$.

We use the following representatives for the right cosets of $B$ in $G$.
\[ r_0=I_2, \quad  r_1=\begin{pmatrix}1&0\\1&1\end{pmatrix}, \quad r_\alpha=\begin{pmatrix}1&0\\\alpha&1\end{pmatrix}, \quad r_{\alpha+1}=\begin{pmatrix}1&0\\\alpha+1&1\end{pmatrix}, \quad \text{and} \quad r_\infty=\begin{pmatrix}0&1\\1&0\end{pmatrix}.\]
We order the corresponding basis of the induced module by
\[
\{r_0\otimes e_1,\ r_0\otimes e_2,\
r_1\otimes e_1,\ r_1\otimes e_2,\
r_\alpha\otimes e_1,\ r_\alpha\otimes e_2,\
r_{\alpha+1}\otimes e_1,\ r_{\alpha+1}\otimes e_2,\
r_\infty\otimes e_1,\ r_\infty\otimes e_2\},
\]
where $\{e_1,e_2\}$ is the standard basis of $\mathbb Q^2$. The induced action is computed using the decomposition $g r_j=r_i b$ for $g \in G$ and $b\in B$. Thus, the $(i,j)$ block of $\Pi_{\mu,\nu}(g)$ is
$\rho_B(r_i^{-1}g r_j)$ if $r_i^{-1}g r_j\in B$, and is zero otherwise.
With this convention, the induced matrices are as follows.

\[
\Pi_{\mu,\nu}(D)=
\begin{pmatrix}
-1&1&0&0&0&0&0&0&0&0\\
-1&0&0&0&0&0&0&0&0&0\\
0&0&0&0&-1&1&0&0&0&0\\
0&0&0&0&-1&0&0&0&0&0\\
0&0&0&0&0&0&-1&1&0&0\\
0&0&0&0&0&0&-1&0&0&0\\
0&0&-1&1&0&0&0&0&0&0\\
0&0&-1&0&0&0&0&0&0&0\\
0&0&0&0&0&0&0&0&0&-1\\
0&0&0&0&0&0&0&0&1&-1
\end{pmatrix},
\]

\[
\Pi_{\mu,\nu}(U)=
\begin{pmatrix}
1&0&0&0&0&0&0&0&0&0\\
0&1&0&0&0&0&0&0&0&0\\
0&0&0&0&0&0&0&0&1&0\\
0&0&0&0&0&0&0&0&0&1\\
0&0&0&0&0&0&0&-1&0&0\\
0&0&0&0&0&0&1&-1&0&0\\
0&0&0&0&-1&1&0&0&0&0\\
0&0&0&0&-1&0&0&0&0&0\\
0&0&1&0&0&0&0&0&0&0\\
0&0&0&1&0&0&0&0&0&0
\end{pmatrix}, \quad \Pi_{\mu,\nu}(W)=
\begin{pmatrix}
0&0&0&0&0&0&0&0&1&0\\
0&0&0&0&0&0&0&0&0&1\\
0&0&1&0&0&0&0&0&0&0\\
0&0&0&1&0&0&0&0&0&0\\
0&0&0&0&0&0&-1&1&0&0\\
0&0&0&0&0&0&-1&0&0&0\\
0&0&0&0&0&-1&0&0&0&0\\
0&0&0&0&1&-1&0&0&0&0\\
1&0&0&0&0&0&0&0&0&0\\
0&1&0&0&0&0&0&0&0&0
\end{pmatrix}.
\]

We now determine the commutant
\[
\mathcal A
=\operatorname{End}_{\mathbb QG}(V_{\Pi_{\mu,\nu}})
=\left\{X\in M_{10}(\mathbb Q):
X\Pi_{\mu,\nu}(g)=\Pi_{\mu,\nu}(g)X
\text{ for all }g\in G\right\}.
\]
Since $G$ is generated by $D,U,W$, it is enough to solve
\[
\begin{aligned}
X\Pi_{\mu,\nu}(D)&=\Pi_{\mu,\nu}(D)X,\\
X\Pi_{\mu,\nu}(U)&=\Pi_{\mu,\nu}(U)X,\\
X\Pi_{\mu,\nu}(W)&=\Pi_{\mu,\nu}(W)X
\end{aligned}
\]
over $\mathbb Q$. The solution space has dimension $4$. One can verify by a direct calculation that
\[
A_0=
\begin{pmatrix}
0&0&0&1&1&-1&-1&0&0&1\\
0&0&1&0&0&-1&-1&1&1&0\\
0&1&0&0&-1&0&1&-1&0&1\\
1&0&0&0&-1&1&0&-1&1&0\\
1&-1&-1&0&0&0&0&1&0&1\\
0&-1&-1&1&0&0&1&0&1&0\\
-1&0&1&-1&0&1&0&0&0&1\\
-1&1&0&-1&1&0&0&0&1&0\\
0&1&0&1&0&1&0&1&0&0\\
1&0&1&0&1&0&1&0&0&0
\end{pmatrix} \in \mathcal A.
\]
Furthermore, we have $A_0^2=4I_{10}$ and $\operatorname{tr}(A_0)=0$. Consequently, $E=\frac{A_0+2I_{10}}{4}$ satisfies $E^2=E$ and has rank $5$. Since $E\in\mathcal A$, its image $W_{\mathcal R}:=\operatorname{Im}(E)
\subseteq V_{\Pi_{\mu,\nu}}$ is a $G$-invariant subspace. By Lemma~\ref{lem:idem}, we get $W_{\mathcal R}\cong V_{\mathcal R}$ and $\dim_{\mathbb Q}W_{\mathcal R}=5$.

Let $b_1,\ldots,b_5$ be the first five columns of $E$. These columns
are linearly independent and form a $\mathbb Q$-basis of
$W_{\mathcal R}$. We now compute the matrices of the restriction of
$\Pi_{\mu,\nu}$ to this invariant subspace with respect to this basis.

For $D$, the action on the basis vectors is given by
\[
\begin{aligned}
\Pi_{\mu,\nu}(D)b_1&=-b_1-b_2,\\
\Pi_{\mu,\nu}(D)b_2&=b_1,\\
\Pi_{\mu,\nu}(D)b_3
&=\frac73b_1+\frac23b_2-\frac43b_3-\frac53b_4-2b_5,\\
\Pi_{\mu,\nu}(D)b_4
&=-\frac53b_1-\frac43b_2+\frac53b_3+\frac43b_4+b_5,\\
\Pi_{\mu,\nu}(D)b_5&=-b_3-b_4.
\end{aligned}
\]
It follows that
\[
\mathcal R(D)=
\begin{pmatrix}
-1&1&\frac73&-\frac53&0\\
-1&0&\frac23&-\frac43&0\\
0&0&-\frac43&\frac53&-1\\
0&0&-\frac53&\frac43&-1\\
0&0&-2&1&0
\end{pmatrix}.
\]

Similarly, expressing $\Pi_{\mu,\nu}(U)b_j$ and
$\Pi_{\mu,\nu}(W)b_j$ in the basis
$\{b_1,\ldots,b_5\}$ gives
\[
\mathcal R(U)=
\begin{pmatrix}
1&0&-\frac43&-\frac13&\frac73\\
0&1&\frac13&-\frac23&\frac23\\
0&0&\frac13&\frac43&-\frac43\\
0&0&\frac53&\frac23&-\frac53\\
0&0&1&1&-2
\end{pmatrix} \quad \text{and} \quad \mathcal R(W)=
\begin{pmatrix}
-\frac43&-\frac13&0&0&-\frac23\\
\frac13&-\frac23&0&0&\frac23\\
\frac13&\frac43&1&0&-\frac13\\
\frac53&\frac23&0&1&\frac13\\
1&1&0&0&1
\end{pmatrix}.
\]
Thus, the matrices obtained above describe the restriction of $\Pi_{\mu,\nu}$ to $W_{\mathcal R}$. Consequently, the representation $\mathcal R:G\longrightarrow \mathrm{GL}_5(\mathbb Q)$
is an irreducible $\mathbb Q$-representation of $G$ corresponding to $V_{\mathcal R}$ and affording the character $\Omega(\chi_{q+1}^{(m,n)})$. Moreover, in this particular example, since $\mathbb Q(\chi_{q+1}^{(m,n)})=\mathbb Q$, the complexification of $\mathcal R$ is isomorphic to the principal series representation $\mathcal R\otimes_{\mathbb Q}\mathbb C
\cong \pi_{\mu,\nu}
=\operatorname{Ind}_B^G(\mu\otimes\nu)$
as a $\mathbb C$-representation of $G$.


\section{Proof of Theorem~\ref{thm:wedderburn GL}}\label{sec:GroupAlgbera}
In this section, we prove Theorem~\ref{thm:wedderburn GL}. The Artin--Wedderburn theorem asserts that every semisimple ring decomposes as a finite direct product of matrix rings over division rings. Furthermore, by the Brauer--Witt theorem, the Wedderburn components of a rational group algebra are Brauer equivalent to cyclotomic algebras (see \cite{Yam}).

Let $G$ be a finite group and let $\chi \in \Irr(G)$. By Schur theory, there exists an irreducible $\Q$-representation $\rho$ of $G$ affording the character $\Omega(\chi)$. In particular, distinct Galois conjugacy classes correspond to distinct irreducible rational representations of $G$. Reiner~\cite[Theorem~3]{Reiner} further describes the structure of the simple components appearing in the Wedderburn decomposition of the rational group algebra $\Q G$ associated with these rational representations (also see~\cite[Theorem 3.3.1]{JR}). Throughout, by the direct sum of representations we mean the usual block-diagonal direct sum of matrices.

\begin{lemma}\cite[Theorem~3]{Reiner}\label{Reiner}
Let $\mathbb{K}$ be a field of characteristic zero, and let $\mathbb{K}^*$ denote its algebraic closure. Suppose that $T$ is an irreducible $\mathbb{K}$-representation of $G$, extended linearly to a representation of $\mathbb{K} G$. Define
\[
A=\{T(x):x\in \mathbb{K} G\}.
\]
Then $A$ is a simple $\mathbb{K}$-algebra, and hence $A\cong \M_n(D)$ for some positive integer $n$ and some division algebra $D$ over $\mathbb{K}$.
Moreover, after extending scalars to $\mathbb{K}^*$, the representation $T$ decomposes as
\[
T\otimes_{\mathbb{K}}\mathbb{K}^*
\cong
m_{\mathbb{K}}(\chi)\bigoplus_{\sigma\in\Gal(\mathbb{K}(\chi)/\mathbb{K})}U^{\sigma},
\]
where $U$ is an irreducible $\mathbb{K}^*$-representation of $G$ affording the character $\chi$, $U^{\sigma}$ affords the character $\chi^{\sigma}$, and $m_{\mathbb{K}}(\chi)$ denotes the Schur index of $\chi$ over $\mathbb{K}$. In this situation, we have
\[
n=\frac{\chi(1)}{m_{\mathbb{K}}(\chi)},
\]
the center of $D$ satisfies $Z(D)\cong \mathbb{K}(\chi)$ and $[D:Z(D)]=\bigl(m_{\mathbb{K}}(\chi)\bigr)^2$.
\end{lemma}

We now proceed to the proof of Theorem~\ref{thm:wedderburn GL}.

\begin{proof}[Proof of Theorem~\ref{thm:wedderburn GL}]
Let $G=\GL_2(q)$, where $q$ is a prime power, and let $\chi \in \irr(G)$. Suppose that $\rho$ is an irreducible $\Q$-representation of $G$ affording the character $\Omega(\chi)$. Denote by $A_{\Q}(\chi)$ the simple component in the Wedderburn decomposition of $\Q G$ corresponding to $\rho$. Then
\[
A_{\Q}(\chi)\cong \M_n(D)
\]
for some $n\in \mathbb{N}$ and some division algebra $D$. By Lemma~\ref{lemma:schurindexGL}, we have $m_{\Q}(\chi)=1$. Moreover, by Lemma~\ref{Reiner}, $[D:Z(D)]=m_{\Q}(\chi)^2$ and $Z(D)=\Q(\chi)$. Hence, we get
\[
D=Z(D)=\Q(\chi).
\]

Now consider
\[
\rho \cong \bigoplus_{i=1}^{l}\rho_i,
\]
where $l=[\Q(\chi):\Q]$, and each $\rho_i$ is an irreducible complex representation of $G$ affording the character $\chi^{\sigma_i}$ for some $\sigma_i\in \Gal(\Q(\chi)/\Q)$. Since $m_{\Q}(\chi)=1$, Lemma~\ref{Reiner} further implies that $n=\chi(1)$.

Furthermore,
\[
\cd(G)=\{1,q-1,q,q+1\}.
\]

By Theorem~\ref{thm:RationRepsGL2}, for each $d\mid (q-1)$, there exists exactly one, up to isomorphism, simple $\Q G$-module of dimension $\varphi(d)$ affording the character $\Omega(\chi^{(n)}_1)$ for some $n$ satisfying
\[
d=\frac{q-1}{\gcd(n,q-1)}.
\]
Here,
\[
\chi^{(n)}_1(1)=1
\qquad \text{and} \qquad
\Q(\chi^{(n)}_1)=\Q(\zeta_d).
\]
Moreover, these exhaust all irreducible $\Q$-representations of $G$ affording the character $\Omega(\chi)$ with $\chi(1)=1$. Therefore, the direct sum of the simple components in the Wedderburn decomposition of $\mathbb{Q}G$ corresponding to the inequivalent irreducible $\mathbb{Q}$-representations of $G$ affording the character $\Omega(\chi)$, where $\chi(1)=1$, is isomorphic to

\[
\bigoplus_{d\mid q-1}\Q(\zeta_d).
\]

Similarly, by Theorem~\ref{thm:RationRepsGL2}, for each $d\mid (q-1)$, there exists exactly one, up to isomorphism, simple $\Q G$-module of dimension $q\varphi(d)$ affording the character $\Omega(\chi^{(n)}_q)$ for some $n$ satisfying
\[
d=\frac{q-1}{\gcd(n,q-1)}.
\]
In this case,
\[
\chi^{(n)}_q(1)=q
\qquad \text{and} \qquad
\Q(\chi^{(n)}_q)=\Q(\zeta_d).
\]
Moreover, these constitute all irreducible $\Q$-representations of $G$ affording the character $\Omega(\chi)$ with $\chi(1)=q$. Hence, the direct sum of the simple components in the Wedderburn decomposition of $\mathbb{Q}G$ corresponding to the inequivalent irreducible $\mathbb{Q}$-representations of $G$ affording the character $\Omega(\chi)$, where $\chi(1)=q$, is isomorphic to
\[
\bigoplus_{d\mid q-1} \M_q(\Q(\zeta_d)).
\]

Next, by Theorem~\ref{thm:RationRepsGL2}, for each $d\mid (q^2-1)$ with $d\nmid (q-1)$, there exists exactly one, up to isomorphism, simple $\Q G$-module of dimension
\[
\frac{q-1}{2}\varphi(d)
\]
affording the character $\Omega(\chi^{(n)}_{q-1})$ for some $n$ satisfying
\[
d=\frac{q^2-1}{\gcd(n,q^2-1)}.
\]
Moreover, we have
\[
\chi^{(n)}_{q-1}(1)=q-1
\qquad \text{and} \qquad
\Q(\chi^{(n)}_{q-1})
=
\Q(\zeta_d+\zeta_d^{\,q}).
\]
These account for all irreducible $\Q$-representations of $G$ affording the character $\Omega(\chi)$ with $\chi(1)=q-1$. Consequently, the direct sum of the simple components in the Wedderburn decomposition of $\mathbb{Q}G$ corresponding to the inequivalent irreducible $\mathbb{Q}$-representations of $G$ affording the character $\Omega(\chi)$, where $\chi(1)=q-1$, is isomorphic to
\[
\bigoplus_{\substack{d\mid q^2-1 \\ d\nmid q-1}}
\M_{q-1}(\Q(\zeta_d+\zeta_d^{\,q})).
\]

Finally, let
\[
U=(\mathbb{Z}/(q-1)\mathbb{Z})^\times
\]
act on the set $X$ of unordered pairs $(m,n)$ with $0\le n<m\le q-2$ by multiplication. Let $\mathcal{O}$ be a set of representatives for the resulting orbits. For each $(m,n)\in \mathcal{O}$, there exists exactly one, up to isomorphism, simple $\Q G$-module of dimension
\[
(q+1)\,[\Q(\zeta_{q-1}^{\,m+n},\zeta_{q-1}^{\,m}+\zeta_{q-1}^{\,n}):\Q]
\]
affording the character $\Omega(\chi^{(m,n)}_{q+1})$, where $\zeta_{q-1}$ denotes a primitive $(q-1)$-th root of unity in $\mathbb{C}$. In this case,
\[
\chi^{(m,n)}_{q+1}(1)=q+1
\]
and
\[
\Q(\chi^{(m,n)}_{q+1})
=
\Q(\zeta_{q-1}^{\,m+n},\zeta_{q-1}^{\,m}+\zeta_{q-1}^{\,n}).
\]
Moreover, these exhaust all irreducible $\Q$-representations of $G$ affording the character $\Omega(\chi)$ with $\chi(1)=q+1$. Therefore, the direct sum of the simple components in the Wedderburn decomposition of $\mathbb{Q}G$ corresponding to the inequivalent irreducible $\mathbb{Q}$-representations of $G$ affording the character $\Omega(\chi)$, where $\chi(1)=q+1$, is isomorphic to
\[
\bigoplus_{(m,n)\in\mathcal{O}}
\M_{q+1}(\Q(\zeta_{q-1}^{\,m+n},\zeta_{q-1}^{\,m}+\zeta_{q-1}^{\,n})).
\]

Collecting the simple components corresponding to all inequivalent irreducible $\Q$-representations of $G$ associated with $\chi\in \irr(G)$ yields the desired Wedderburn decomposition of $\Q G$. This completes the proof of Theorem~\ref{thm:wedderburn GL}.
\end{proof}

We conclude this section with Example~\ref{examp:ex}, which illustrates the explicit computation of the Wedderburn decomposition of $\Q G$ for $G=\GL_2(q)$ using Theorem~\ref{thm:wedderburn GL}. This decomposition has also been verified using the \textsc{Wedderga} package of {\sf GAP}~\cite{Gap}.

\begin{example}\label{examp:ex}
Let $G=\GL_2(8)$, so that $q=8$. Let $U=(\mathbb{Z}/7\mathbb{Z})^\times$ act on the set $X$ of unordered pairs $(m,n)$ with $0 \le m, n \le 6$ and $m \neq n$ by multiplication modulo $7$. Then
\[
\mathcal{O}=\{(1,0),(2,1),(3,1),(6,1)\}
\]
is a complete set of orbit representatives. Hence, by Theorem~\ref{thm:RationRepsGL2}, we have
\begin{align*}
\Q G \cong {}& \Q \oplus \Q(\zeta_7)
\oplus \M_7(\Q(\zeta_3+\zeta_3^8))
\oplus \M_7(\Q(\zeta_9+\zeta_9^8)) \\
&\oplus \M_7(\Q(\zeta_{21}+\zeta_{21}^8))
\oplus \M_7(\Q(\zeta_{63}+\zeta_{63}^8))
\oplus \M_8(\Q) \\
&\oplus \M_8(\Q(\zeta_7))
\oplus \M_9(\Q(\zeta_7,\zeta_7+1))
\oplus \M_9(\Q(\zeta_7^2,\zeta_7^2+\zeta_7)) \\
&\oplus \M_9(\Q(\zeta_7^3,\zeta_7^3+\zeta_7))
\oplus \M_9(\Q(1,\zeta_7^6+\zeta_7)).
\end{align*}

Note that $\zeta_{21}+\zeta_{21}^8=\zeta_{21}(1+\zeta_{21}^7)=\zeta_{21}(1+\zeta_3)$. Since $1+\zeta_3=e^{\pi i/3}$, we obtain
\[
\zeta_{21}+\zeta_{21}^8
=e^{2\pi i/21}e^{\pi i/3}
=e^{3\pi i/7}
=\zeta_{14}^3.
\]
Hence, $\Q(\zeta_{21}+\zeta_{21}^8)=\Q(\zeta_{14})=\Q(\zeta_7)$, since $\zeta_{14}^2=\zeta_7$. Therefore, we get
\begin{align*}
\Q G \cong {}& \Q \oplus \Q(\zeta_7)
\oplus \M_7(\Q)
\oplus \M_7(\Q(\zeta_9+\zeta_9^8))
\oplus \M_7(\Q(\zeta_7)) \\
&\oplus \M_7(\Q(\zeta_{63}+\zeta_{63}^8))
\oplus \M_8(\Q)
\oplus \M_8(\Q(\zeta_7)) \\
&\oplus 3\M_9(\Q(\zeta_7))
\oplus \M_9(\Q(\zeta_7+\zeta_7^6)).
\end{align*}
This expression agrees with the Wedderburn decomposition of $\Q G$, computed by the \textsc{Wedderga} package of {\sf GAP}~\cite{Gap}.
\end{example}

\section*{Acknowledgments} We are grateful to Prof. Anupam Singh and Prof. B. Sury for their valuable discussions and insightful suggestions during the preparation of this manuscript.






\printbibliography
 \vspace{2em}
\end{document}